\documentclass{jpp}
\usepackage{graphicx}
\usepackage{amsmath}
\usepackage{url}
\usepackage{color}

\newtheorem{theorem}{Theorem}

\newcommand{\change}[1]{{#1}}

\begin{document}

\title{A study on conserving invariants of the Vlasov equation in semi-Lagrangian computer simulations}
\shorttitle{Conservation of invariants of the Vlasov equation}
\author{L. Einkemmer\corresp{\email{lukas.einkemmer@uibk.ac.at}}}
\affiliation{University of Innsbruck, Innsbruck, Austria}
\maketitle
\begin{abstract}
The semi-Lagrangian discontinuous Galerkin method, coupled with a splitting approach in time, has recently been introduced for the Vlasov--Poisson equation. Since these methods are conservative, local in space, and able to limit numerical diffusion, they are considered a promising alternative to more traditional semi-Lagrangian schemes.
In this paper we study the conservation of important physical invariants and the long time behavior of the semi-Lagrangian discontinuous Galerkin method. To that end we conduct a theoretical analysis and perform a number of numerical simulations. In particular, we find that the entropy is nondecreasing for the discontinuous Galerkin scheme, while unphysical oscillations in the entropy are observed for the traditional method based on cubic spline interpolation. 
\end{abstract}

\section{Introduction}

The so-called semi-Lagrangian methods constitute a class of numerical
schemes used to discretize hyperbolic partial differential equations
(usually first order equations). The basic idea is to follow the characteristics
backward in time. For the Vlasov--Poisson equation the characteristics
corresponding to the splitting sub-steps can be determined analytically
(as has been suggested in the seminal paper of \citet{cheng1976}).
However, since the endpoint of a characteristic curve does not necessarily
coincide with the grid used, an interpolation procedure has to be
employed. An obvious choice is to reconstruct the desired function
by spline interpolation, which according to \citet{sonnendruecker2011}
is still considered the de facto standard in Vlasov simulations. A
downside of this approach is that a tridiagonal linear system of equations
has to be solved to construct the spline. This algorithm has a low
flop/byte ratio and significant communication overhead (both of which
are unfavorable on most modern and future high performance computing
systems).

On the other hand, the semi-Lagrangian discontinuous Galerkin method
employs a piecewise polynomial approximation in each cell of the computational
domain \citep{qiu2011,rossmanith2011,crouseilles2011,einkemmer2014}.
In case of the advection equation the discretized function is translated
and then projected back to the appropriate subspace of piecewise polynomial
functions. This method, per construction, is mass conservative and
only accesses two adjacent cells in order to compute the necessary
projection (this is true independent of the order of the approximation).
Furthermore, from the literature available it seems that the semi-Lagrangian
discontinuous Galerkin method compares favorable with spline interpolation.
In addition, mathematical rigorous convergence results are available
\citep{einkemmer2014convergence,einkemmer2014}.

\change{A further method that has been widely employed in plasma simulations is the
van Leer scheme (see, for example, \cite{fijalkow1999, mangeney2002, galeotti2005, valentini2005, califano2006, galeotti2006}).
This case is formally similar to the discontinuous Galerkin scheme employed in this paper. What
distinguishes the discontinuous Galerkin method from the van Leer scheme is that the coefficients in the corresponding basis
expansion are stored directly in computer memory. In contrast, the van Leer scheme replaces these coefficients 
by performing an approximation using suitable differences on an equidistant grid.}

The semi-Lagrangian discontinuous Galerkin scheme is fully explicit
(i.e.~no linear system has to be solved to advance the solution in
time) and thus it is easier to implement (especially on parallel architectures)
and shows a more favorable communication pattern. We should note that
some measures have been taken to improve the parallel scalability
of the cubic spline interpolation \citep{crouseilles2009}.
However, even for this approach a relatively large communication overhead
is incurred. This is due to the fact that the boundary condition for
the local spline reconstruction requires a large stencil if the desirable
properties of the global cubic spline interpolation are to be preserved
(the method derived in \citet{crouseilles2009} requires a centered
stencil of size $21$). 

It is well known that the Vlasov--Poisson equation conserves a number
of physically important variables. In many applications it is important
that these quantities are conserved by the discretization (at least
up to some tolerance; ideally up to machine precision) for long time
intervals. For the cubic spline interpolation, which is employed in
a number of software packages that are used to simulate the Vlasov
equation, the good long time behavior is well established (even though
there are no rigorous results). \change{Furthermore, the conservation of invariants
has been studied extensively for the (mostly second and third order) van Leer schemes.
In \cite{mangeney2002} the diffusive and dispersive error has been investigated and compared
to spline interpolation. In \cite{galeotti2005} and \cite{califano2006} particular emphasis
has been placed on the diffusive behavior of the van Leer schemes and spline interpolation
(which is of paramount importance for collisionless plasma simulations).}

While some short time numerical results with respect to the conservation
properties are available for the discontinuous Galerkin semi-Lagrangian
method \citep{qiu2011,crouseilles2011}, to the
knowledge of the author no systematic study has been performed. Thus,
the purpose of this paper is to investigate the performance of the
semi-Lagrangian discontinuous Galerkin scheme with respect to the
conservation of a number of physically important variables. We will
also present long time integration results which so far have, to our
knowledge, been absent from the literature. In addition, we investigate
what effect the degree of the polynomial approximation has in this
context. This is particular important for the discontinuous Galerkin
approach as choosing a higher order approximation reduces the number
of cells available in the numerical simulation.

In section \ref{sec:eq-numerical-methods} we introduce the Vlasov--Poisson
equation, the splitting approach used for the time discretization,
and both the semi-Lagrangian discontinuous Galerkin method as well
as the semi-Lagrangian method based on the cubic spline interpolation.
In section \ref{sec:conservation} we analyze the discontinuous Galerkin
approach from a theoretical point of view and compare it to the cubic
spline interpolation. This allows us to better understand the numerical
results that have been obtained in section \ref{sec:numer-results}.
In this section we provide results that compare the discontinuous
Galerkin approach with the cubic spline interpolation and perform
long time simulations in order to show the viability of the semi-Lagrangian
discontinuous Galerkin scheme for the present application. Finally,
we conclude in section \ref{sec:Conclusion}.

\section{Equations and numerical methods\label{sec:eq-numerical-methods}}

In astro- and plasma physics the behavior of a collisionless plasma
is modeled by the Vlasov equation
\[
\partial_{t}f(t,x,v)+v\cdot\nabla f(t,x,v)+F\cdot\nabla_{v}f(t,x,v)=0,
\]
which is posed in an up to $3+3$ dimensional phase space (although
lower dimensional models are often employed as well). The variable
$x$ denotes the position and variable $v$ denotes the velocity.
The density function $f$ is the sought-after particle distribution
and the (force) term $F$ describes the interaction of the plasma
with the electromagnetic field. Depending on the application this
force term can include the full Lorentz force or only the force due
to the electric field.

In this paper we will restrict ourselves to the Vlasov--Poisson equation,
i.e. the force term is given by $F=E$ and the electric field is determined
by solving 
\[
\nabla\cdot E=\rho(x)-1,\qquad\nabla\times E=0
\]
with charge density
\[
\rho(x)=\int f(t,x,v)\,\mathrm{d}v.
\]
In addition, in all our simulations we impose periodic boundary conditions
in both the $x$- and $v$-directions.

To advance the numerical solution in time we use the splitting approach
introduced in \citet{cheng1976}. That is, we consider (the first part
of the splitting algorithm)
\begin{equation}
\partial_{t}f(t,x,v)+v\cdot\nabla f(t,x,v)=0,\qquad f(0,x,v)=g(x,v)\label{eq:splitting-part-1}
\end{equation}
and denote the corresponding solution at time $\tau$ by 
\[
\mathrm{e}^{\tau A}g(x,v)=f(\tau,x,v).
\]
For the second part of the splitting algorithm we then consider
\begin{equation}
\partial_{t}f(t,x,v)+E^{n}\cdot\nabla_{v}f(t,x,v)=0,\qquad f(0,x,v)=g(x,v),\label{eq:splitting-part-2}
\end{equation}
where the electric field $E^{n}$ is determined from the density 
\[
\rho^{n}(x)=\int g(x,v)\,\mathrm{d}v
\]
and is thus taken constant during each time step. We denote the corresponding
solution by 
\[
\mathrm{e}^{\tau B}g(x,v)=f(\tau,x,v).
\]
Note that, contrary to what the notation suggests, $B$ is a nonlinear
operator as the electric field depends on the density function $f$.
Using the introduced notation we can easily formulate a time step
of the splitting algorithm. The Lie splitting
\[
f^{n+1}(x,v)=\mathrm{e}^{\tau B}\mathrm{e}^{\tau A}f^{n}(x,v),
\]
where $f^{n}$ is an approximation of $f(t_{n},x,v)$ and $\tau=t_{n+1}-t_{n}$
is the time step size, is a first order method. In all our simulation
we will use the Strang splitting scheme
\[
f^{n+1}(x,v)=\mathrm{e}^{\frac{\tau}{2}A}\mathrm{e}^{\tau B}\mathrm{e}^{\frac{\tau}{2}A}f^{n}(x,v)
\]
which is a second order method and has roughly the same computational
cost as the Lie splitting scheme. Let us also note that high order
splitting methods have been constructed \citep{crouseilles2011high}. 

The main computational advantage of the splitting scheme is that the
resulting sub-steps, given by (\ref{eq:splitting-part-1}) and (\ref{eq:splitting-part-2}),
are in the form of an advection equation, where the advection speed
is independent of the variable in the direction of the advection.
Thus, we can immediately solve (\ref{eq:splitting-part-1}) to obtain
\begin{equation}
\mathrm{e}^{\tau A}g(x,v)=g(x-\tau v,v)\label{eq:translation-1}
\end{equation}
and (\ref{eq:splitting-part-2}) to obtain
\begin{equation}
\mathrm{e}^{\tau B}g(x,v)=g(x,v-\tau E(x)),\label{eq:translation-2}
\end{equation}
where $E$ is the electric field determined from the charge density
corresponding to $g(x,v)$. Thus, by using the splitting approach
outlined above we have reduced the task of computing a numerical approximation
to the Vlasov--Poisson equation to computing two translations in phase
space.

Now, up to this point phase space is still continuous. However, in
order to implement the numerical scheme on a computer, we have to
perform a space discretization that represents the numerical solution,
at a fixed time step, using a finite number of degrees of freedom.
The most straightforward approach is to choose a uniform grid. However,
for any grid $(x_{i},v_{j})$ the translations $x_{i}-\tau v_{j}$
and $v_{j}-\tau E(x_{i})$, in general, do not coincide with a grid
point. Thus, we have to use an interpolation scheme in order to evaluate
$g(x_{i}-\tau v_{j})$ and $g(v_{j}-\tau E(x_{i}))$. In the Vlasov
community interpolation based on cubic splines is very popular. This
method is illustrated in figure \ref{fig:pol-interp} and proceeds
in two steps:
\begin{enumerate}
\item construct a cubic spline using the known value of the function on
the grid points (this involves the solution of a tridiagonal system
of equations);
\item translate each grid point according to (\ref{eq:translation-1}) or
(\ref{eq:translation-2}) and use the cubic spline to obtain the new
value at that grid point.
\end{enumerate}
Since we follow the characteristics backward in time (i.e.~track
particles along their trajectories) but use an Eulerian grid for the
space discretization, these methods are usually referred to as semi-Lagrangian
and are free of a Courant--Friedrichs--Lewy (CFL) condition. In addition,
they do not exhibit the numerical noise that is prevalent in particle
methods (such as the popular particle in cell scheme). 

\begin{figure}
\begin{centering}
\includegraphics[width=8cm]{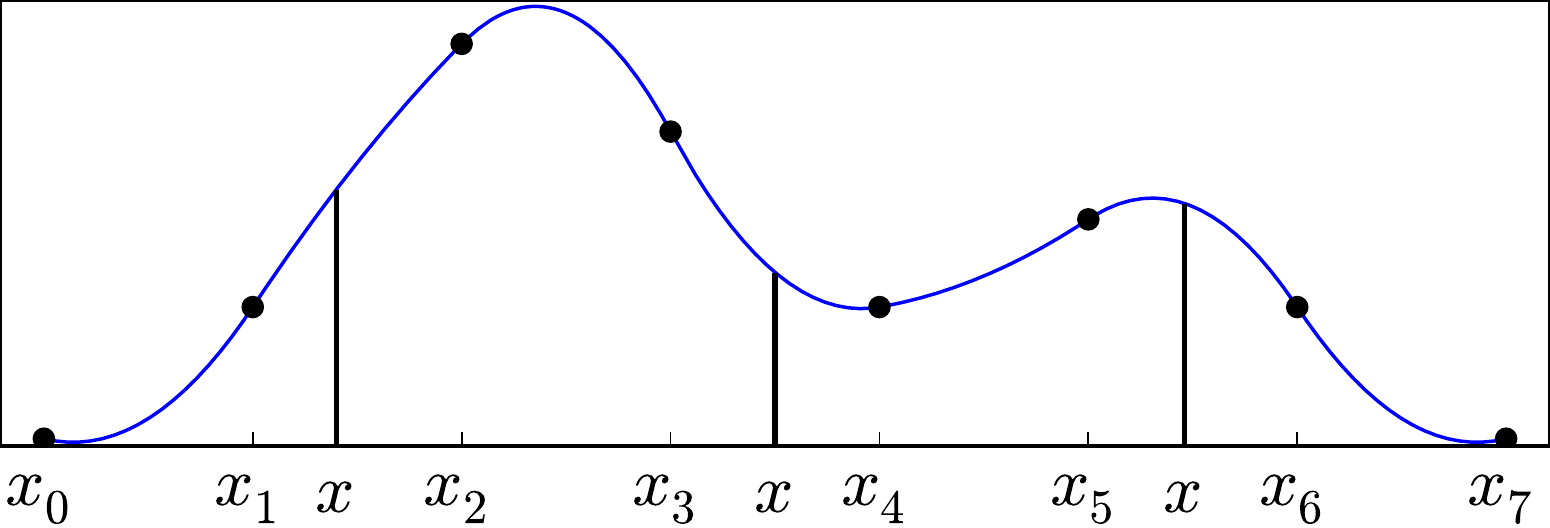}
\par\end{centering}

\caption{An illustration of the semi-Lagrangian method based on a quadratic
spline interpolation is shown. In this case a continuous piecewise
polynomial interpolant is constructed that is used for all evaluations.
We illustrate three evaluations of the interpolant (denoted by the
variable $x$). \label{fig:pol-interp}}
\end{figure}

For the Vlasov--Poisson equation it is vital that charge (or mass;
the two quantities are equivalent as all particles in the physical
model carry the same amount of electrical charge) is conserved. It
is, a priori, not clear that the a semi-Lagrangian scheme based on
spline interpolation satisfies this constraint. In order to show this,
we reformulate the numerical method as a forward semi-Lagrangian scheme
(as is done in \citet{crouseilles2009forward}). Since our interpolant
is a cubic spline, we can expand $f^{n}(x,v)$ in the B-spline basis
(for simplicity we restrict ourselves to the 1+1 dimensional case
here)
\[
f^{n}(x,v)=\sum_{kl}\omega_{kl}^{n}S(x-x_{k})S(v-v_{l}),
\]
where $S$ is given by
\[
6S(x)=\begin{cases}
4-6(x/h)^{2}+3\vert x/h\vert^{3} & 0\leq\vert x\vert\leq h\\
(2-\vert x/h\vert)^{3} & h\leq\vert x\vert\leq2h\\
0 & \text{otherwise}
\end{cases}
\]
and $h$ is the grid spacing. The coefficients $\omega_{kl}^{n}$
are determined uniquely by the function values at the grid points
but in order to obtain them we have to solve a linear system of equations.
The forward semi-Lagrangian method based on cubic spline interpolation
for (\ref{eq:translation-1}) is then given by
\[
f^{n+1}(x_{i},v_{j})=\sum_{kl}\omega_{kl}^{n}S(x_{i}-\tau v_{j}-x_{k})S(v_{j}-v_{l}).
\]
For the constant advection case considered here this is merely a reformulation
of the previously introduced (backward) semi-Lagrangian method. However,
the interpretation is different (translating the basis function vs
following the characteristics backward in time). That the forward
semi-Lagrangian scheme is mass conservative can be easily deduced
from the present formulation. First, we note that the B-splines form
a partition of unity. That is, for all $x$ we have
\[
\sum_{k}S(x-x_{k})=1.
\]
Therefore, we can write the integral over a spline $u$ as follows
\[
\int u(x)\,\mathrm{d}x=\sum_{k}\omega_{k}\int S(x-x_{k})\,\mathrm{d}x=h\sum_{k}\omega_{k}=h\sum_{i}\sum_{k}\omega_{k}S(x_{i}-x_{k})=h\sum_{i}u(x_{i}),
\]
which shows that the discrete charge is in fact equal to the continuous
charge given by the cubic spline interpolation. Finally, we have\textbf{
\[
\sum_{i}u(x_{i}-\tau v)=\sum_{i}\sum_{k}\omega_{k}S(x_{i}-\tau v-x_{k})=\sum_{k}\omega_{k}\sum_{i}S(x_{i}-\tau v-x_{k})=\sum_{k}\omega_{k}=\sum_{i}u(x_{i})
\]
}which shows that the semi-Lagrangian scheme based on the cubic spline
interpolation is charge conservative (up to machine precision). We
further remark that, assuming a sufficient regular solution, this
method is fourth order accurate in space.

We now turn our attention to the semi-Lagrangian discontinuous Galerkin
scheme. To that end we divide our domain into cells $C_{i}=I_{i}^{(1)}\times\dots\times I_{i}^{(d)}$,
where the $I_{i}$ are one-dimensional intervals of length $h$, and
assume that a function $u$ is given such that $u\vert_{C_{i}}$,
i.e. the restriction of $u$ to the $i$th cell, is a polynomial of
degree $k$. Then the function $u$ lies in the approximation space
(note that we do not enforce a continuity constraint across cell interfaces).
However, in general, this is not true for the translated function
\[
T_{\tau}u(x)=u(x-\tau v),
\]
where we have introduced the translation operator $T_{\tau}$. Thus,
we perform an approximation by applying a projection operator $P$
to obtain
\[
u^{n+1}=PT_{t}u^{n}.
\]
Note that except for this projection no approximation has been made
so far. The function $u^{n+1}$ constitutes the sought-after approximation
of $u^{n}(x-\tau v)$. The operator $P$ is the $L^{2}$ projection
on the (finite dimensional) subspace of cell-wise polynomials of degree
$\ell$. This approach is illustrated in figure \ref{fig:illustration-dg}.

\begin{figure}
\centering{}\includegraphics[width=6cm]{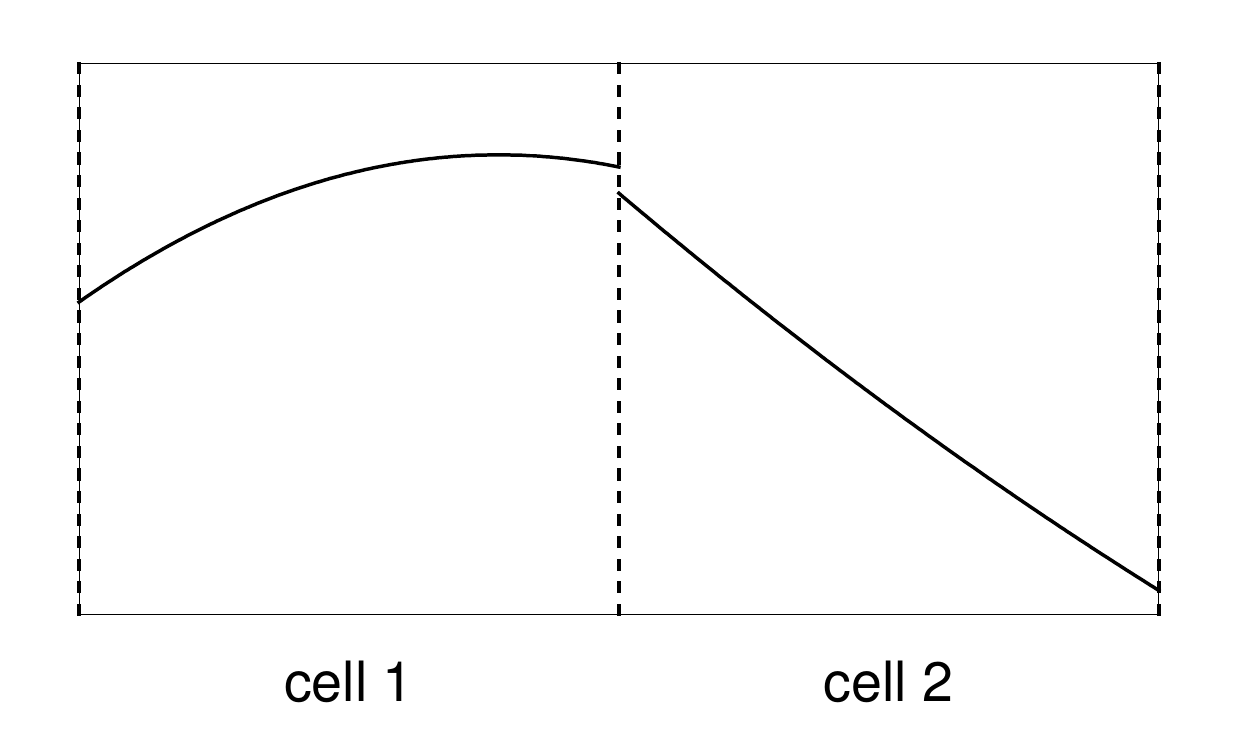}\includegraphics[width=6cm]{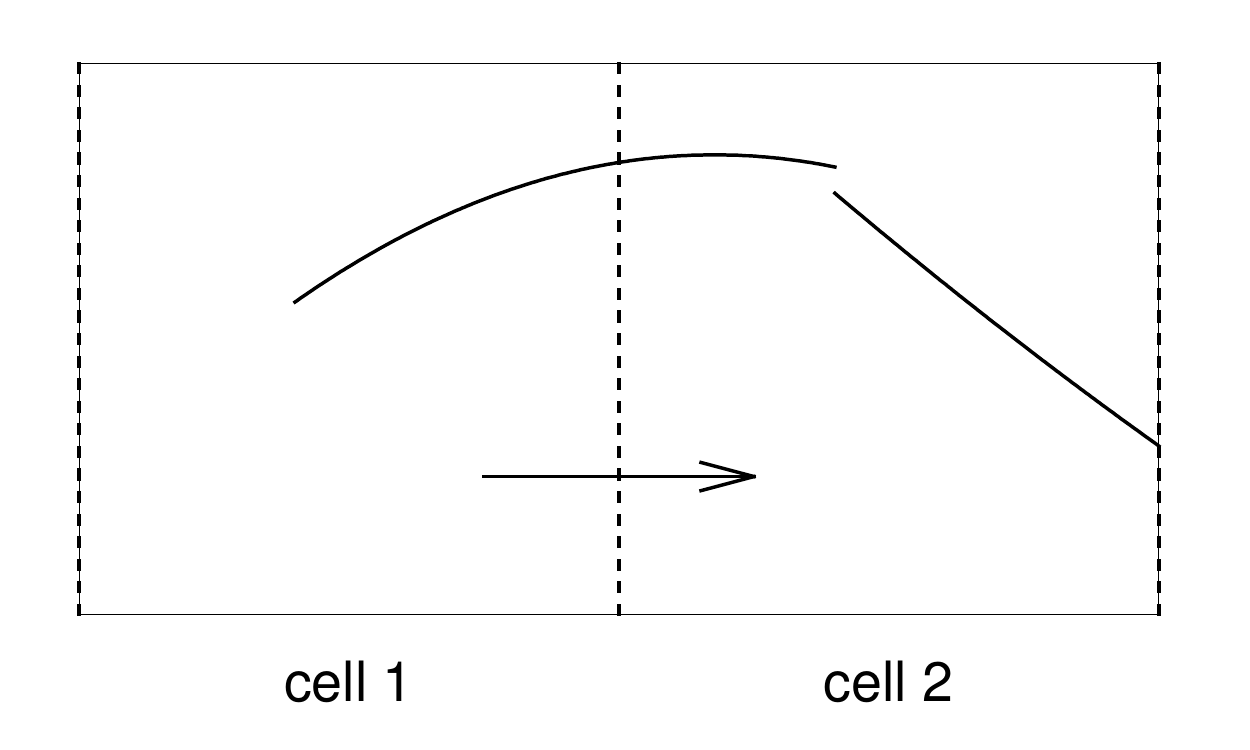}\caption{Illustration of the semi-Lagrangian discontinuous Galerkin approximation
in a single space dimension. The discontinuous piecewise polynomial
approximation is translated in space and then projected back to the
finite dimensional subspace of piecewise polynomial functions. \label{fig:illustration-dg}}
\end{figure}
In order to implement the numerical scheme, we still have to choose
a basis for the space of cell-wise polynomials of degree $\ell$.
This then also determines the degrees of freedom stored in computer
memory. Our implementation is based on the Lagrange basis at the Gauss--Legendre
quadrature nodes. Thus, the degrees of freedom correspond to the value
of the function at the Gauss--Legendre nodes in each cell. However,
in certain situations, it is more convenient to express the piecewise
polynomial function in the Legendre basis. Thus, we can write (for
an $L^{2}$ function $u$) 
\[
u(x)\approx Pu(x)=\sum_{i=0}^{N-1}\sum_{j=0}^{\ell}u_{ij}p_{ij}(x),
\]
where $p_{ij}$ is the (appropriately scaled) $j$th Legendre polynomials
in the $i$th cell and $N$ is the number of cells. The coefficients
in the expansion are then given by
\[
u_{ij}=\int u(x)p_{ij}(x)\,\mathrm{d}x.
\]
Since the Legendre polynomials form an orthogonal basis, we have
\[
\int u(x)\,\mathrm{d}x=\sum_{i}u_{i0}=\int Pu(x)\,\mathrm{d}x
\]
and thus we follow that the numerical scheme is charge conservative.

Before proceeding, let us make two remarks. First, the semi-Lagrangian
discontinuous Galerkin method is fully explicit. That is, no linear
system of equations has to be solved. Second, the order of the space
discretization is $p+1$ and thus we can easily construct numerical
schemes of arbitrary order. Note, however, that the degrees of freedom
are directly proportional to the order of the method. We will revisit
this issue in some detail in section \ref{sec:numer-results}. 

Up to this point we have described the numerical schemes used to solve
the Vlasov equation. However, what we have not discussed so far is
how to solve Poisson's equation in order to determine the electric
field from the charge density. Since the electric field is usually
a smooth function, we can use the fast Fourier transformation (FFT).
This is straightforward for the cubic spline interpolation as the
degrees of freedom already reside on an equidistant grid. However,
for the discontinuous Galerkin approach the degrees of freedom are
located at the Gauss--Legendre quadrature nodes and thus have to be
transferred to an equidistant grid before an FFT algorithm can be
applied. In our implementation we simply evaluate the piecewise polynomial
in order to obtain an equidistant grid. For example, for a piecewise
linear approximation and the cell $[x_{i-1/2},x_{i-1/2}+h]$ we evaluate
the approximant at $x_{i-1/2}+\tfrac{1}{3}h$ and $x_{i-1/2}+\tfrac{2}{3}h$
which yields an equidistant grid. In addition, once the electric field
has been computed it has to be transferred back from an equidistant
grid to the Gauss--Legendre quadrature nodes. In our implementation
we use a polynomial interpolation that is of the same order as the
discontinuous Galerkin scheme. Since this part of the algorithm takes
place in a lower dimensional space, we could, in principle, use a
more elaborate interpolation scheme without significantly increasing
the computational cost. In sections \ref{sec:conservation} and \ref{sec:numer-results}
we will further investigate the effect of this interpolation on the
conservation properties of the numerical scheme.

\section{Conservation \label{sec:conservation}}

The purpose of this section is to analyze the conservation properties
from a theoretical point of view. The conserved variables (for the
continuous system) which we will consider here are: charge, $L^{1}$
norm of the density function, current, energy, entropy, and the $L^{2}$
norm of the density function. We will first consider the error that
is made due to the time discretization (the splitting approach outlined
in the introduction) and then consider the error made due to the discontinuous
Galerkin or the cubic spline interpolation that are used to discretize
space. An overview of the properties of each scheme is given in table
\ref{tab:conservation}.

\begin{table}
\begin{centering}
\begin{tabular}{ccccc}
\hline 
conserved var. & splitting & dG & spline & remarks\tabularnewline
\hline 
charge & yes & yes & yes & \tabularnewline
current & yes & $\text{yes}^{\dagger}$ & $\text{yes}^{\dagger}$ & additional error term for $\ell=0$\tabularnewline
energy & no & no & no & additional error term for $\ell=0$, $\ell=1$\tabularnewline
entropy & yes & no & no & \tabularnewline
$L^{1}$ & yes & $\text{no}^{\ast}$ & no & \tabularnewline
$L^{2}$ & yes & no & no & \tabularnewline
\hline 
\end{tabular}
\par\end{centering}

\caption{This table provides an overview of the conserved quantities for the
splitting approach in time and the two space discretization strategies.
The superscript $\dagger$ means the the conservation is only true
up to an error made in the computation of the electric field. The
superscript $\ast$ emphasizes that positivity preserving semi-Lagrangian
discontinuous Galerkin schemes have been developed that do preserve
the $L^{1}$ norm as well as the mass. The degree of the polynomial
in each cell (for the sldg scheme) is denoted by $\ell$ and the corresponding
numerical method is consistent of order $\ell+1$ in space. \label{tab:conservation}}
\end{table}

\subsection{Time discretization}

\subsubsection*{Charge and $L^{p}$ norms}

Since we can write both parts of the splitting as a translation along
a coordinate axis in phase space, we immediately conclude that the
splitting conserves the mass as well as all $L^{p}$ norms.

\subsubsection*{Current}

More interesting is the conservation of the (total) current. The current
is defined by the following relation
\[
j=\int vg(x,v)\,\mathrm{d}(x,v).
\]
For the first part of the splitting (see (\ref{eq:translation-1}))
we immediately have
\[
\int vg(x-vt,v)\,\mathrm{d}(x,v)=\int vg(x,v)\,\mathrm{d}(x,v)
\]
which is the desired conservation of the current. For the second part
of the splitting (see (\ref{eq:translation-2})) we get
\begin{align*}
\int vg(x,v-E(x)t)\,\mathrm{d}(x,v) & =\int(v+E(x)t)g(x,v)\,\mathrm{d}(x,v)\\
 & =\int vg(x,v)\,\mathrm{d}(x,v)+t\int E(x)\rho(x)\,\mathrm{d}x.
\end{align*}
Then from 
\begin{align*}
\int E\rho\,\mathrm{d}x & =\int E(\nabla\cdot E)\,\mathrm{d}x\\
 & =\int E(\nabla\cdot E)+\tfrac{1}{2}\nabla E^{2}\,\mathrm{d}x\\
 & =-\int(\nabla\times E)\times E\,\mathrm{d}x\\
 & =0
\end{align*}
the conservation of the current follows immediately. Note that we
have assumed here that the electric field is computed exactly and
that it is curl free. Both of these assumptions are obviously satisfied
if we neglect the space discretization. However, we will have to consider
these assumptions in more detail in the next section.

\subsubsection*{Energy}

The total energy is given by
\[
E=\int\tfrac{1}{2}v^{2}g(x,v)\,\mathrm{d}(x,v)+\int\tfrac{1}{2}E^{2}(x)\,\mathrm{d}x,
\]
where the first term represents the kinetic energy and the second
term represents the energy stored in the electric field. The first
part of the splitting conserves the kinetic energy
\[
\int v^{2}g(x-tv,v)\,\mathrm{d}(x,v)=\int v^{2}g(x,v)\,\mathrm{d}(x,v)
\]
and the second part of the splitting conserves the electric energy
(since the charge density is not changed). However, in both cases
the reverse statement is not true. Thus, the second part of the splitting
does not conserve the electric energy and the first part of the splitting
does not conserve the kinetic energy. Most importantly, the complete
splitting procedure does not conserve the total energy. 

This is to be expected. The splitting scheme considered here is a
Hamiltonian splitting (see, for example, \citet{casas2015}). Such
numerical methods have been studied extensively in the case of ordinary
differential equations (see, for example, \citet{hairer2006}). While
these numerical methods do not conserve energy, they show excellent
long time properties with respect to energy conservation. The results
for ordinary differential equations can not be directly transferred
to the present case. The hope is, however, that similar long time
properties can also be observed for the Vlasov--Poisson equation.

\subsubsection*{Entropy}

The entropy is defined as
\begin{equation}
S=-\int g\log g\,\mathrm{d}(x,v).\label{eq:entropy}
\end{equation}
For physically reasonable initial values (i.e.~$g\geq0$) the entropy
is well defined and conserved by the continuous system. Entropy is
maximized for a density function which has equal probability for each
point in phase space. Since diffusive systems naturally tend to this
limit, we use entropy as a measure of the amount of diffusion that
is introduced by the numerical scheme. 

To analyze the splitting scheme, we define $f(t,x,v)=g(x-tv,v)$.
Then, for the first part of the splitting algorithm, we have
\begin{align*}
\partial_{t}\int f\log f\,\mathrm{d}(x,v) & =-\int(\log f)\nabla\cdot(vf)\,\mathrm{d}(x,v)-\int\nabla\cdot(vf)\,\mathrm{d}(x,v).
\end{align*}
Integration by parts yields 
\[
\int(\log f)\nabla\cdot(vf)\,\mathrm{d}(x,v)=-\int fv\cdot\nabla\log f\,\mathrm{d}(x,v)=-\int\nabla\cdot(vf)\,\mathrm{d}(x,v)=0
\]
and thus entropy is conserved for the first part of the splitting
algorithm. 

Now, for the second step in the splitting algorithm we define $f(t,x,v)=g(x,v-tE(x))$
and get
\[
\partial_{t}\int f\log f\,\mathrm{d}(x,v)=-\int(\log f)\nabla_{v}\cdot(Ef)\,\mathrm{d}(x,v)-\int\nabla_{v}\cdot(Ef)\,\mathrm{d}(x,v)=0
\]
which shows that entropy is conserved for the splitting scheme.

\subsection{Space discretization}

\subsubsection*{Charge and $L^{1}$ norm}

The discontinuous Galerkin method is charge conservative (see section
\ref{sec:eq-numerical-methods}). The same is true for the spline
interpolation. However, both numerical methods can produce negative
values. This, in principle, is not an issue for the algorithm considered
here. The only difficulty that might arise is the physical interpretation
of the resulting particle density function. Note, however, that as
long as the magnitude of the negative values are smaller than the
accuracy required, this issue is essentially void. However, an interesting
aspect is that with respect to conservation of the $L^{1}$ norm we
can formulate the following result. 
\begin{theorem}
A numerical algorithm that produces negative values and is mass conservative
can not conserve the $L^{1}$ norm.\end{theorem}
\begin{proof}
We assume that the initial value $f(0,x,v)$ is non-negative. Then
\[
\int\vert f(0,x,v)\vert\,\mathrm{d}(x,v)=\int f(0,x,v)\,\mathrm{d}(x,v)=\int f(t,x,v)\,\mathrm{d}(x,v)\leq\int\vert f(t,x,v)\vert\,\mathrm{d}(x,v).
\]
The inequality becomes strict if the particle density function $f(t,x,v)$
is negative. In this case the $L^{1}$ norm is not conserved which
is the desired result.
\end{proof}

The above result shows that conservation of $L^{1}$ norm can be used
as a measure of positivity preservation. Both the discontinuous Galerkin
scheme and the method based on cubic splines, at least in their most
basic formulation, can yield negative values. Thus, they do not preserve
the $L^{1}$ norm. For the discontinuous Galerkin scheme (mass conservative)
limiters are available that restrict the density function in such
a manner that negative values are avoided \citep{rossmanith2011,qiu2011}. The resulting schemes are then both mass and
$L^{1}$ conservative. One should note, however, that while the proposed
positivity limiters can be applied locally, associated to them is
a certain computational cost. What is even more important is that
these procedure might degrade the performance of the scheme with respect
to the other conserved quantities. For example, the limiters suggested
in \citet{rossmanith2011} and \citet{qiu2011} operate by shifting
and scaling the output from the translation in order to avoid negative
values. However, in order to conserve mass, a positive shift must
result in a compression of the polynomial function which in turn introduces
additional diffusion into the numerical scheme. 

On the other hand, for the cubic spline interpolation implementing
a positivity limiter is significantly more complicated. Nevertheless,
a number of limiters have been suggested in the literature \citep{zerroukat2005,zerroukat2006,crouseilles2010}. It
has been well established in this case that adding positivity limiters
to the numerical scheme degrades the conservation of the other invariants.

We will not focus on positivity limiters in this paper. However, we
will investigate the appearance of negative values in some detail
in section \ref{sec:numer-results}. In fact, the numerical results
obtained seem to suggest that even without such modifications we observe
good long time behavior for the semi-Lagrangian discontinuous Galerkin
scheme.

\subsubsection*{Current}

In the following, we will use $g(x,v)=(Pf(\cdot-vt,v))(x)$. From
section \ref{sec:eq-numerical-methods} we know that both the discontinuous
Galerkin and the cubic spline based scheme are mass conservative and
thus
\[
\int g(x,v)\,\mathrm{d}x=\int f(x-vt,v)\,\mathrm{d}x.
\]
For the first part of the splitting algorithm we thus have
\[
\int vg(x,v)\,\mathrm{d}(x,v)=\int v\left(\int g(x,v)\,\mathrm{d}x\right)\mathrm{d}v=\int v\left(\int f(x-vt,v)\,\mathrm{d}x\right)\mathrm{d}v=\int vf(x,v)\mathrm{\,d}(x,y),
\]
which implies that the current is conserved for that step of the algorithm. 

The situation is more involved for the second part of the splitting
algorithm. This is due to the fact that the translation is now in
the velocity direction. We will thus consider the two numerical schemes
separately, starting with the discontinuous Galerkin method. To simplify
the notation let us define $g(x,v)=(Pf(x,\cdot-tE(x)))(v)$. Expanding
$v$ in the Legendre basis (in a given cell) yields
\[
v=v_{0}p_{0}+v_{1}p_{1}(v),
\]
where $p_{k}$ is the $k$th Legendre polynomial in the corresponding
cell. Thus, we get 
\begin{align*}
\int vg(x,v)\,\mathrm{d}(x,v) & =\int\left(\int(v_{0}p_{0}+v_{1}p_{1}(v))g(x,v)\,\mathrm{d}v\right)\mathrm{d}x\\
 & =\int\left[v_{0}p_{0}\int f(x,v-tE(x))\,\mathrm{d}v+v_{1}\int p_{1}(v)g(x,v)\,\mathrm{d}v\right]\,\mathrm{d}x.
\end{align*}
Now, since we can expand the function $v\mapsto g(x,v)$ into
\begin{equation}
g(x,v)=\sum_{k=0}^{\ell}g_{k}(x)p_{k}(v)\label{eq:expansion}
\end{equation}
we have (for $\ell\geq1$) 
\[
\int p_{1}(v)g(x,v)\,\mathrm{d}v=g_{1}(x)=\int p_{1}(v)f(x,v-tE(x))\,\mathrm{d}x
\]
and thus
\[
\int vg(x,v)\,\mathrm{d}(x,v)=\int vf(x,v-tE(x))\,\mathrm{d}(x,v).
\]
The only remaining issue is that we do not use the exact electric
field in our computation but replace it by a numerical approximation.
In principle, we can use an FFT based approach to compute the electric
field up to machine precision. However, the discontinuous Galerkin
approach requires an interpolation procedure in order to transfer
the data from and to the equidistant grid required to compute the
FFT. This transformation is not exact up to machine precision and
does introduce an error in the total current. We will investigate
this issue in section \ref{sec:numer-results}.

Now, let us consider the semi Lagrangian method based on the cubic
spline interpolation. That the first part of the splitting algorithm
conserves the current is an immediate consequence of the fact that
that scheme is mass conservative. For the second part of the splitting
algorithm we have 
\[
\sum_{j}v_{j}(Pf(x,\cdot-tE(x)))(v_{j})=\sum_{j}v_{j}f(x,v_{j}-tE(x))
\]
since the projection does not change the value of the spline at the
grid points. Thus, we follow that the cubic spline interpolation preserves
the current under the assumption that the electric field is determined
up to machine precision. Note that satisfying this assumption is easier
in the case of the spline based semi-Lagrangian scheme, as the degrees
of freedom already reside on an uniform grid and can thus be directly
fed into an FFT based Poisson solver.

\subsubsection*{Energy}

We have seen in the previous section that the splitting scheme does
not conserve the total energy. Thus, the question posed in this section
is if the space discretization introduces an additional error source.

First, let us consider the discontinuous Galerkin scheme. For the
first part of the splitting we define $g(x,v)=(Pf(\cdot-vt,x))(x)$
and obtain \change{(using conservation of mass)
\[
\int v^{2}g(x,v)\,\mathrm{d}(x,v)=\int v^{2}\left(\int g(x,v)\,\mathrm{d}x\right)\,\mathrm{d}v=\int v^{2}\left(\int f(x-vt,v)\,\mathrm{d}x\right)\,\mathrm{d}v.
\]
Now, we can rewrite the integrals on the right-hand side as follows
\[ \int v^{2}\left(\int f(x-vt,v)\,\mathrm{d}x\right)\,\mathrm{d}v
    = \int v^{2}f(x-vt,v)\,\mathrm{d}(x,v)
   \]
which shows that the kinetic energy is preserved.} On the other hand,
the first splitting step does not conserve the electric energy. Thus,
we consider here only if the space discretization modifies the electric
energy that we obtain from the time integrator. This is indeed the
case as in general 
\[
\int f(x-vt,v)\,\mathrm{d}v\neq\int g(x,v)\,\mathrm{d}v
\]
which shows that the space discretization modifies the charge density
and thus the electric field. Therefore, the space discretization introduces
an additional error source in the electric energy. 

For the second part of the splitting we define $g(x,v)=(Pf(x,\cdot-tE(x)))(v)$
and obtain in each cell (by using (\ref{eq:expansion})) 
\begin{align*}
\int v^{2}g(x,v)\,\mathrm{d}(x,v) & =\int\left(\sum_{k=0}^{2}\sum_{j=0}^{\ell}v_{k}g_{j}(x)\int p_{k}(v)p_{j}(v)\,\mathrm{d}v\right)\,\mathrm{d}x\\
 & =\int\sum_{k=0}^{\min(2,\ell)}v_{k}g_{k}\,\mathrm{d}x\\
 & =\int\left(\sum_{k=0}^{\min(2,\ell)}v_{k}\int p_{k}(v)f(x,v-tE(x))\,\mathrm{d}v\right)\,\mathrm{d}x\\
 & =\int v^{2}f(x,v-tE(x))\,\mathrm{d}(x,v).
\end{align*}
The last equality only holds for $\ell\geq2$ and shows that in this
case no additional error in the kinetic energy is introduced by the
space discretization. Now, the second part of the splitting does not
change the charge density and thus the electric energy remains invariant.
Since
\[
\rho(x)=\int f(x,v-tE(x))\,\mathrm{d}v=\int g(x,v)\,\mathrm{d}v,
\]
the same property holds true for the discontinuous Galerkin approximation
and thus the electric energy is conserved during that step (i.e.~no
additional error is introduced by the space discretization).

Most of the results obtained for the semi-Lagrangian discontinuous
Galerkin scheme can be transferred immediately to the case of spline
interpolation. Since the projection on the spline subspace does not
change the value of the approximant at the grid points, we have 
\[
\sum_{j}v_{j}^{2}g(x,v_{j})=\sum_{j}v_{j}^{2}f(x,v_{j}-tE(x))
\]
and thus the second part of the splitting does not introduce an additional
source of error (independent of the order of the underlying polynomials).

\subsubsection*{Entropy and $L^{2}$ norm}

For both the semi-Lagrangian discontinuous Galerkin scheme and the
cubic spline based scheme the entropy and the $L^{2}$ norm are not
conserved; even though both are conserved by the splitting algorithm.
We will evaluate the relative error made in the conservation of these
quantities in the next section. However, for now, we will discuss
one additional issue not directly related to conservation.

In the Vlasov--Poisson equations no diffusion mechanism is included.
Thus, the $L^{2}$ norm and the entropy are conserved for the continuous
model. However, once diffusion is added, the $L^{2}$ norm will decrease
in time. This is due to the fact that diffusion decreases the variance
of the density function $f$ and we can write $\Vert f\Vert_{2}^{2}=1+\text{Var}(f)$.
In addition, diffusion increases the entropy (which is a measure of
the disorder of a system). The latter property is, of course, just
a statement of the second law of thermodynamics. Thus, the decrease
in the $L^{2}$ norm and the increase in entropy can be considered
as a measure of the amount of diffusion in a given physical system.
In our case they are used to quantify the amount of diffusion that
a numerical scheme (wrongly) introduces. This point is interesting
because in the numerical simulations conducted in the next section
we observe that the semi-Lagrangian discontinuous Galerkin scheme
always decreases/increases the $L^{2}$ norm/entropy while the spline
based approach shows an oscillating behavior for some problems. Thus,
the spline based semi-Lagrangian method cannot be considered as the
exact solution of a perturbed physical system (as no such system would
be allowed to decrease the entropy; a consequence of the second law
of thermodynamics).

Analytically, we can show that the $L^{2}$ norm decreases in time
for the semi-Lagrangian discontinuous Galerkin scheme.
\begin{theorem}
The $L^{2}$ norm is a decreasing function of time for the semi-Lagrangian
discontinuous Galerkin scheme.\end{theorem}
\begin{proof}
Due to the structure of the Vlasov--Poisson equation it is sufficient
to consider a one-dimensional advection problems. Thus, we consider
a translated function $g(x)$ which is piecewise polynomial and thus
lies in $L^{2}$. We can then find a Legendre series
\[
\sum_{k=0}^{\infty}g_{k}p_{k}(x)
\]
which converges in the $L^{2}$ norm to $g$ \citep{pollard1972}.
Now, we have
\[
\Vert Pg\Vert_{2}^{2}-\Vert g\Vert_{2}^{2}=\sum_{k=0}^{\ell}g_{k}^{2}-\sum_{k=0}^{\infty}g_{k}^{2}\leq0
\]
which implies that the $L^{2}$ norm decreases in each time step.
\end{proof}

Based on an extensive numerical investigation, we conjecture that
a similar result holds true for the entropy. In fact, for piecewise
constant polynomials this immediately follows from the convexity of
the logarithm and for piecewise linear polynomials a similar (although
more tedious) argument can be given. However, we were not able to
prove the corresponding result for piecewise polynomials of arbitrary
degree.

There is one additional complication in case of the entropy. As the
numerical scheme can generate negative values, the entropy as given
by (\ref{eq:entropy}) is not even defined. Since we interpret negative
values as zero particle density, we use
\[
S=-\int\max(g,0)\log g\,\mathrm{d}(x,v)
\]
instead. Of course, using this formula can still result in a decrease
of entropy. However, we will see in the next section that this issue
only plays a negligible role in the numerical simulations.

\section{Numerical results\label{sec:numer-results}}

\change{Before we discuss the results of the medium and long time simulations, which
are the main focus of this paper, let us start by establishing some baseline results
for relatively small final times in case of the linear Landau damping (see section \ref{sec:ll} for more details on the parameters used). In this case we
have computed the error for both the electric field $E$ and the particle distribution $f$
at final time $T=12.5$ and $T=100$. The corresponding results are shown in Figure \ref{fig:ll-error}.}

\change{For final time $T=12.5$ we observe convergence of the discontinuous Galerkin schemes with the expected
order. In addition, the sixth order scheme is always superior to the fourth order scheme which in turn
is superior to the second order scheme. The comparison with the cubic spline interpolation is a bit more nuanced.
In particular, we see a clear advantage of the fourth and sixth order discontinuous Galerkin methods for the
error in the electric field. On the other hand, with respect to the error in the distribution function the cubic spline
interpolation performs best (except for very stringent accuracy requirements).}

\change{For the final time $T=100$ the analytic solution is almost identical to zero. However, numerical methods have significant
difficulties due to the recurrence effect (see section \ref{sec:ll} for more details). We can see this in the numerical
results very clearly as initially no convergence is observed (up to $128$ degrees of freedom). Since the recurrence effect
depends mostly on the grid spacing, the spline interpolation performs best in this case (this is due to the fact that an
equidistant grid minimizes the grid spacing). We also note that for low tolerances the second order dG method performs better than
the higher order variants if one is only interested in resolving the error in the electric field. Some of this conclusions have been
pointed our in earlier work. For example, it was observed in \cite{crouseilles2011,einkemmer2014} that higher order discontinuous Galerkin methods only have a weak influence on the recurrence effect.
}

\begin{figure}
    \begin{center}
        \hfill
        \includegraphics[width=6.5cm]{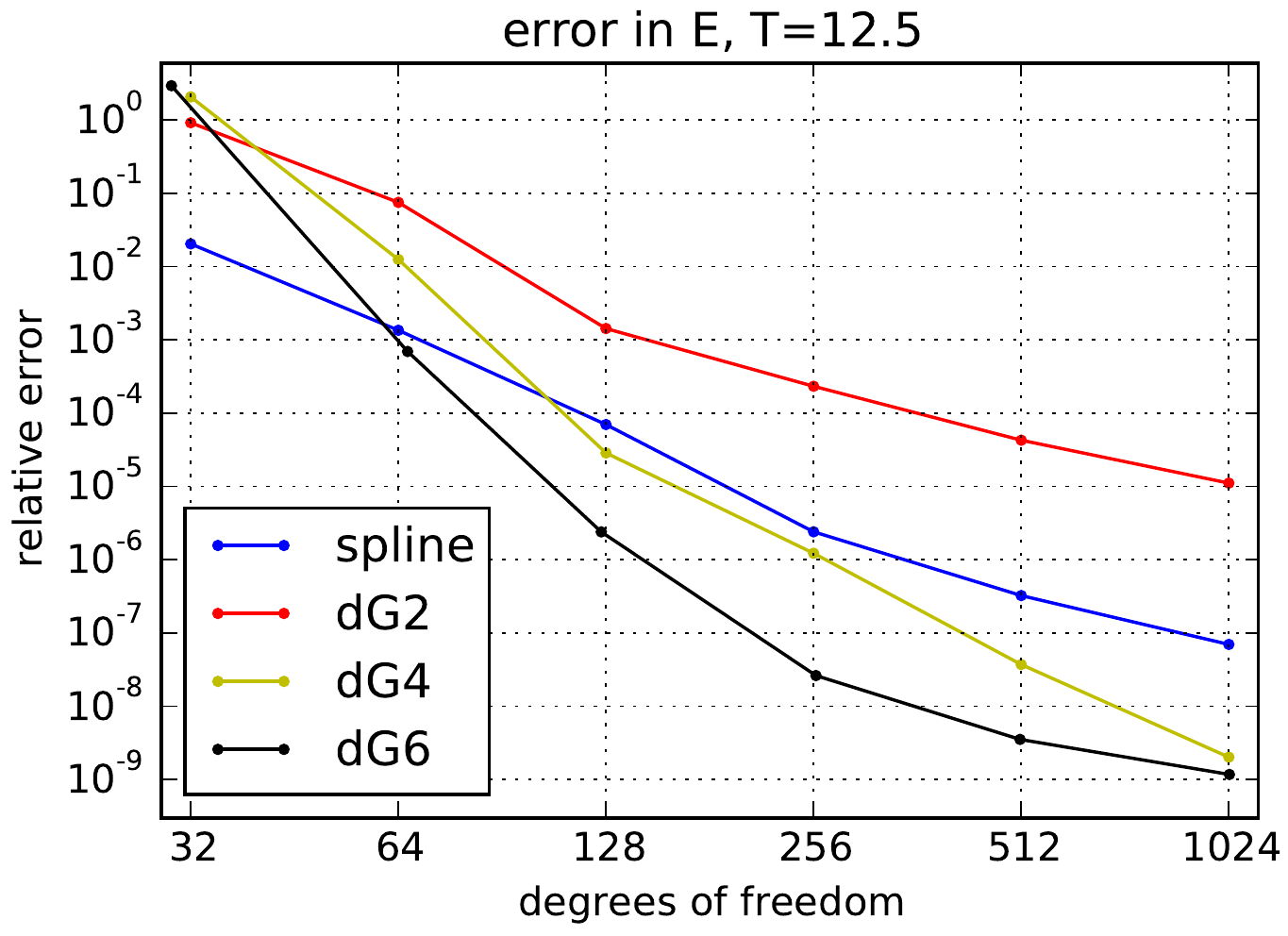}
        \hfill
        \includegraphics[width=6.5cm]{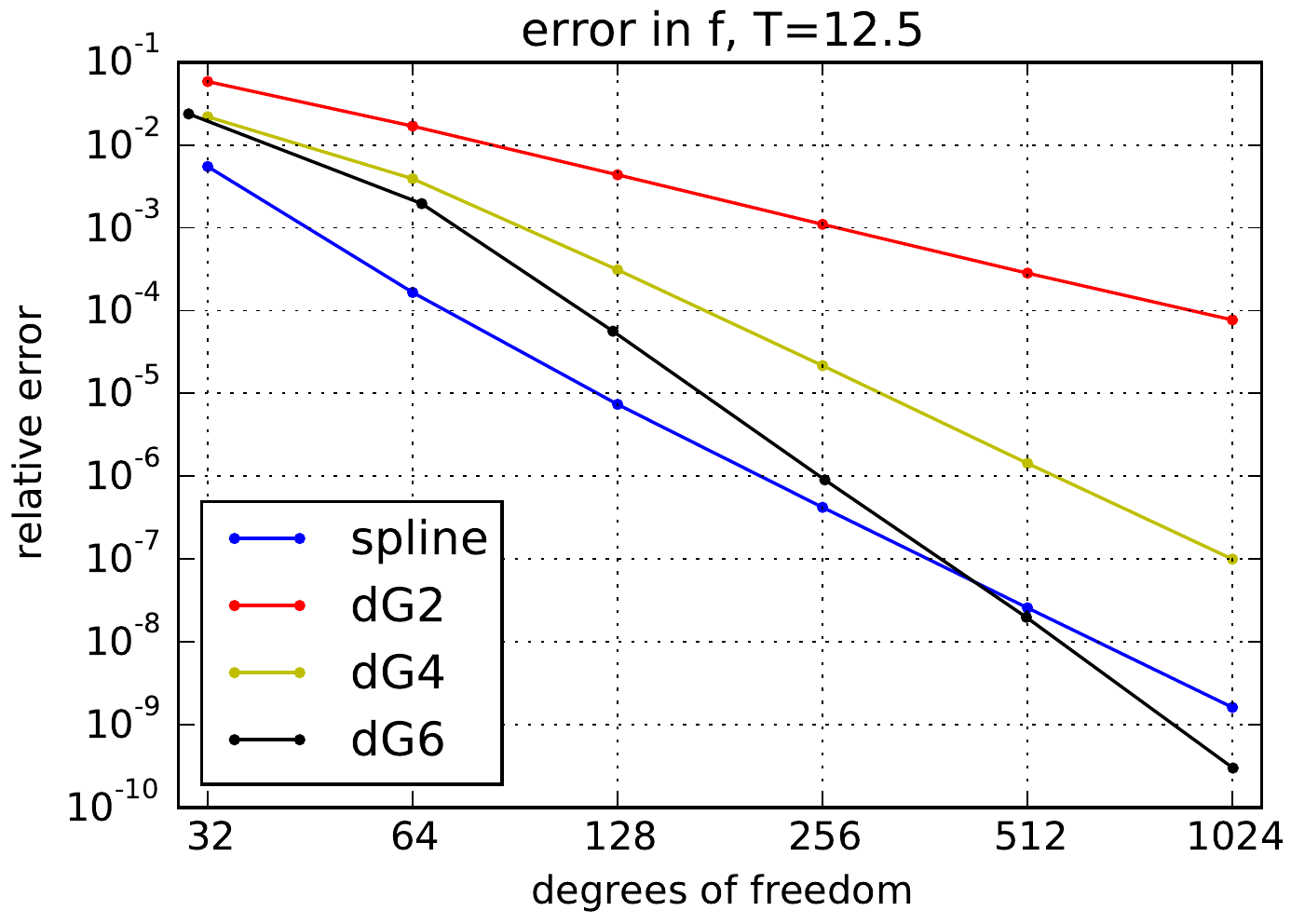}
        \hfill
    \end{center}
    
    \begin{center}
        \hfill
        \includegraphics[width=6.5cm]{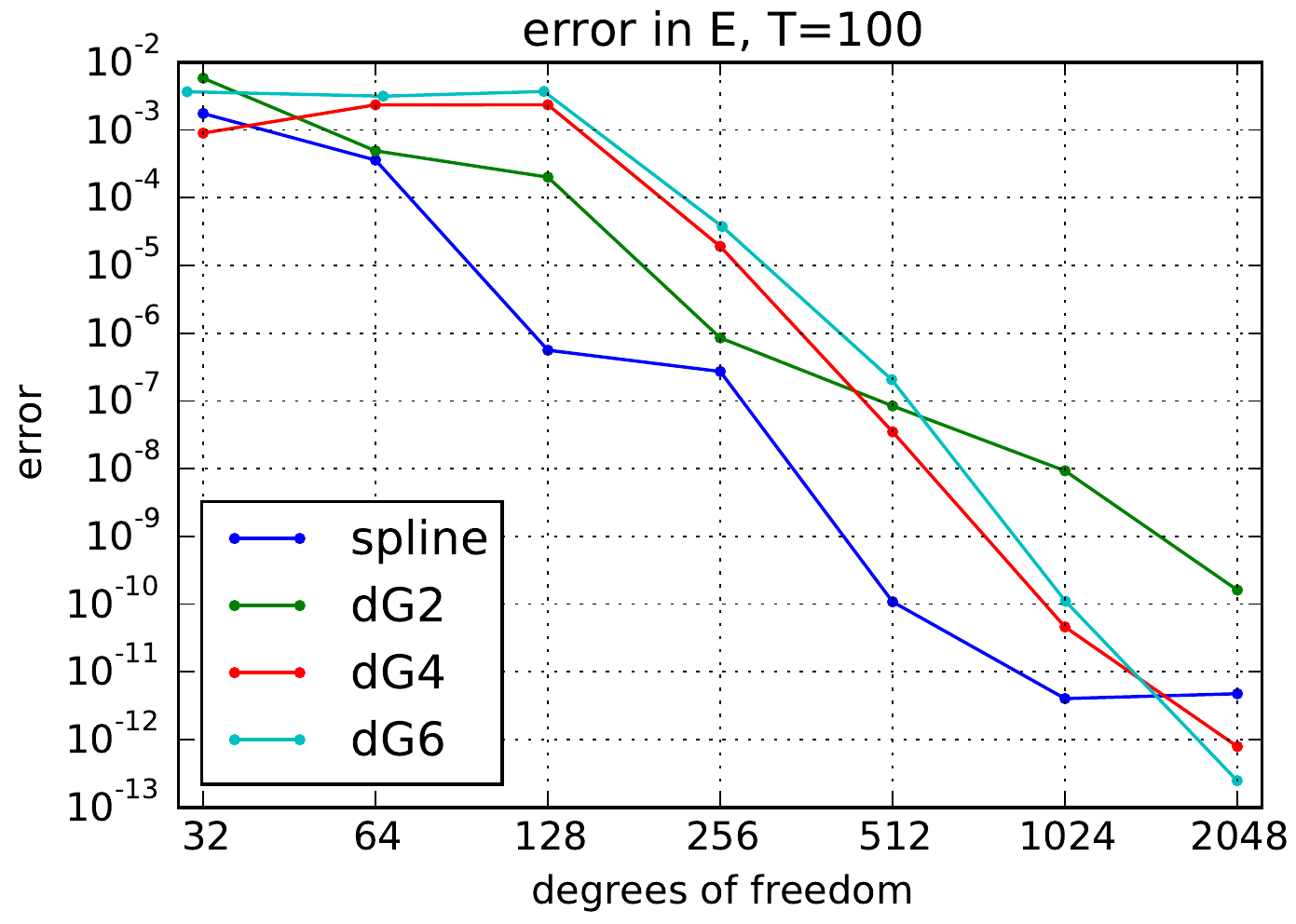}
        \hfill
        \includegraphics[width=6.5cm]{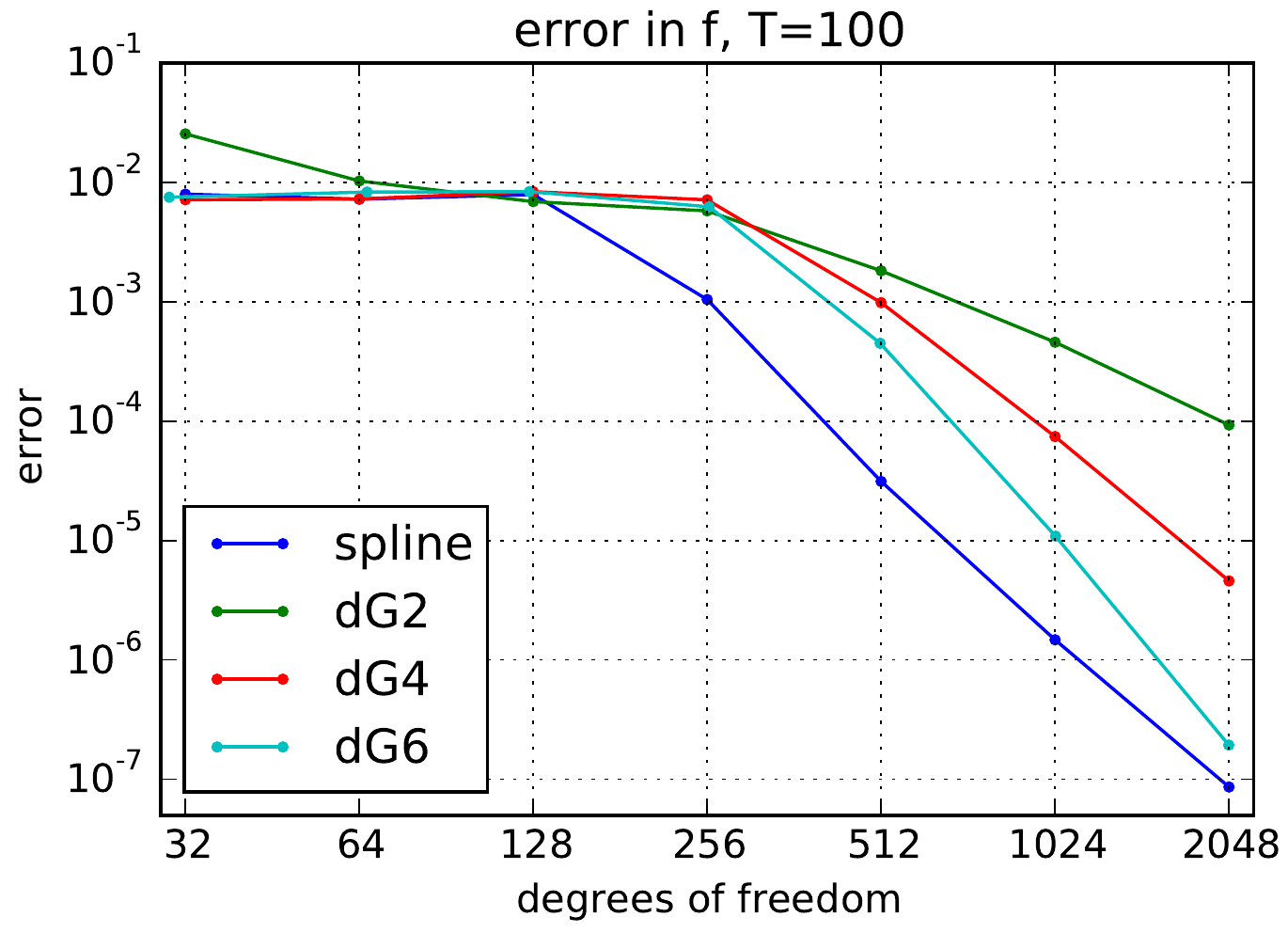}
        \hfill
    \end{center}
    \caption{\change{The error for linear Landau damping in the electric field $E$ and the particle distribution $f$ is shown as a function of
    the degrees of freedom (per dimension) for the second, fourth, and sixth order discontinuous Galerkin method and the cubic spline interpolation. The results for final time $T=12.5$ are shown in the top row and the results for final time $T=100$ are shown in the bottom row.} \label{fig:ll-error} }
\end{figure}

\change{The remainder of this} section is divided into two parts. In the first part we mainly
discuss the properties of the semi-Lagrangian discontinuous Galerkin
scheme and compare its efficiency to the cubic spline interpolation.
We will consider both classic test problems (linear and nonlinear
Landau damping, bump-on-tail instability) and a somewhat less well
known test problem (expansion of a plasma blob, plasma echo). In the
second part we will consider the long time behavior of the semi-Lagrangian
discontinuous Galerkin scheme. These simulations are conducted with
the highly parallel implementation described in \citet{einkemmer2015}.

Since one of our goals is to compare the discontinuous Galerkin scheme
with the cubic spline interpolation, a measure of the cost of a given
numerical scheme has to be established. In a sequential implementation
both computational efficiency as well as memory consumption is a possible
choice. However, comparing arithmetic operations on a modern (vectorized
and multi-core) CPU requires a very delicate analysis that depends
on both algorithmic and implementation details (for example, taking
into account how the linear system is solved in the cubic spline interpolation
or whether such a scheme could be vectorized). In addition, on a distributed
memory system (i.e.~on a cluster), especially for the cubic spline
interpolation, further problems occur due to the all-to-all communication
required in solving the linear system. Some remedies have been proposed
in \citet{crouseilles2009}. However, these usually
change the numerical scheme in a non-trivial way (which might also
have a detrimental effect on the conserved quantities). In fact, this
problem is one of the major motivations why we consider the discontinuous
Galerkin approach.

\change{
In the present study we will exclusively use the number of degrees of freedom,
i.e.~the memory consumed, as the measure of efficiency. This is
reasonable because for higher dimensional simulations the amount of memory consumed
is often the most important criterium for Vlasov simulations. But perhaps even more
important is the fact that such algorithms are bandwidth limited on all modern
architectures. Consequently, how many bytes have to be transferred to and from the CPU
determines the run time of the algorithm. We have implemented our discontinuous Galerkin
approach on a dual socket Intel Xeon E5-2630 v3 workstation with a (theoretical) memory bandwidth of 59 GB/s.
For the second order implementation (dG2)/fourth order (dG4) implementation  we achieve
approximately $94$\%/ $89$\% of the theoretically possible performance on that machine. That means
that for the advection step with contiguous memory access we require approximately $0.3$ ns per degree of freedom.
Let us also note that this means, in particular, that no spline implementation is able to outperform
our algorithm by more than $10$\% (in all likelihood, even an optimized spline implementation
will probably be significantly slower; due to the fact that extra data structures are required to
store the spline coefficients). We further note that the discontinuous Galerkin
method has been implemented on graphic processing units (GPUs) and the Intel Xeon Phi. The corresponding 
performance measurements, which show a significant speedup on both of these architectures, are given in \cite{einkemmer2016}.
}

In the remainder of this section we will conduct two-dimensional simulations
of the Vlasov--Poisson equations using $128$ degrees of freedom in
each dimension (the coarse problem) and $512$ degrees of freedom
in each dimension (the fine problem).

\subsection{Nonlinear Landau damping\label{sub:nl}}

We begin with the classic problem of nonlinear Landau damping. The
initial value
\[
f(0,x,v)=\frac{1}{\sqrt{2\pi}}\mathrm{e}^{-v^{2}/2}(1+\alpha\cos kx)
\]
is posed on the domain $[0,4\pi]\times[-6,6]$ with $k=\tfrac{1}{2}$
and $\alpha=\tfrac{1}{2}$. Periodic boundary conditions are imposed
in both the $x$- and the $v$-direction. A large body of physics
literature has been devoted to this problem and a number of numerical
simulations have been conducted (see, for example, \citet{brunetti2000}
and \citet{manfredi1997}). As such the behavior, at least for moderate
times, is fairly well understood. Initially we observe a decay in
the electric energy (the Landau damping phenomenon) that after some
time stabilizes and gives way to oscillations of the electric energy.
The nonlinear Landau damping is numerically challenging because even
though the initial value is perfectly smooth, small scale structures
(so-called filaments) develop in \change{the two-dimensional} phase space as the system evolves
in time. Thus, in order to integrate the system up to a given precision
an extremely high resolution in velocity space is needed. However,
what we are interested in is to what extend a relatively coarse discretization
can give qualitatively correct results.

The numerical simulations for the cubic spline and the discontinuous
Galerkin approach are given in figure \ref{fig:nl-128} (for the coarse
discretization) and in figure \ref{fig:nl-512} (for the fine discretization).
We observe that for the discontinuous Galerkin method the error in
the current is close to machine precision (even though this quantity
is not exactly conserved). The cubic spline interpolation has an advantage
with respect to the energy error (which, however, is quite accurately
conserved for this test problem). We also observe a distinct difference
between the second order method ($p=1$) and the higher order methods
with respect to energy conservation (as predicted by the results derived
in section \ref{sec:conservation}).

On the other hand, the discontinuous Galerkin scheme has an advantage
with respect to the $L^{1}$ norm (an error of 5\% vs 1-2\% for the
coarse problem). Thus, the discontinuous Galerkin scheme produces
fewer negative values. If we consider the qualitative features of
the numerical solution we might be tempted to conclude that the spline
interpolation reproduces more clearly the features of the exact solution.
However, upon closer inspection of the results for the fine discretization
(where the maximal amplitude is significantly reduced for the spline
interpolation) and taking into account the error in the $L^{1}$ norm,
we conclude that in fact the spline interpolation produces significant
overshoot into the numerical simulation. 

Let us also discuss the effect of the order of the discontinuous Galerkin
approximation. For the coarse problem the order does only play a minor
role (although energy conservation is better once we consider a numerical
scheme with piecewise polynomials of degree $2$ or above). Other
than that the performance of the numerical scheme is almost independent
of the order.

\subsection{Bump-on-tail instability\label{sub:bot}}

The second problem we consider is the so-called bump-on-tail instability.
In this case we prescribe the initial value
\[
f(0,x,v)=\frac{1}{\sqrt{2\pi}}\left(\alpha\mathrm{e}^{-v^{2}/2}+\beta\mathrm{e}^{-4(v-2.5)^{2}}(1+\gamma\cos x)\right)
\]
on the domain $[0,4\pi]\times[-6,6]$ and with $\alpha=0.8$, $\beta=0.2$
and $\gamma=0.1$. Periodic boundary conditions in both the $x$-
and the $v$-direction are imposed. Thus, we have a stationary thermal
plasma in which $80$\% of the mass is concentrated. In addition,
a smaller beam with average velocity $v=2.5$ penetrates the stationary
plasma. This configuration is an equilibrium of the Vlasov--Poisson
equation. However, due to the perturbation that is added, a traveling
vortex appears in phase space. The bump-on-tail instability shows
a quiet phase that is succeeded by the relative rapid development
of the vortex which then remains stable for a long time. Since the
bump-on-tail instability exhibits chaotic behavior the onset of the
vortex might be different for different numerical schemes (even for
relatively fine discretizations). 

The numerical simulations for the cubic spline and the discontinuous
Galerkin approach are given in figure \ref{fig:bot-128} (for the
coarse discretization) and in figure \ref{fig:bot-512} (for the fine
discretization). The current in case of the discontinuous Galerkin
scheme is only conserved up to $10^{-5}$ which clearly confirms the
behavior discussed in section \ref{sec:conservation}. However, this
fact does not seem to diminish the overall performance of the numerical
scheme. For the higher order discontinuous Galerkin scheme and for
the cubic spline interpolation all the remaining conserved quantities
show a similar behavior. The exception being the second order discontinuous
Galerkin scheme which exhibits significantly worse performance across
the board (contrary to the results obtained in the previous section
for the nonlinear Landau damping). 

Let us remark that for both schemes the time evolution of the electric
energy is almost indistinguishable from the finer space discretization
once the instability has been saturated.

\subsection{Expansion into a uniform ion background}

In this section we consider the following initial value
\[
f(0,x,v)=\frac{1}{2\pi}\mathrm{e}^{-v^{2}/2}\mathrm{e}^{-(x-2\pi)^{2}/2}
\]
on the domain $[0,4\pi]\times[-6,6]$. As before, we impose periodic
boundary conditions in both directions. This is the only test problem
where the initial value is localized in space as well as in velocity.
This problem can be interpreted as the expansion of an electron blob
(in thermodynamic equilibrium) into a region of uniform ion density.
This is a challenging problem because there is no physical mechanism
that would hold the initial blob in place. Thus, each particle with
a given velocity expands at its own speed which due to the periodicity
of the problem yields a very irregular phase space distribution (similar
to the nonlinear Landau damping). What sets this problem apart from
the nonlinear Landau damping, however, is that the filamentation is
driven by the free streaming part of the Vlasov equation and is not
merely superimposed on top of a Maxwellian distribution in velocity
space. 

The numerical simulations for the cubic spline and the discontinuous
Galerkin approach are given in figure \ref{fig:vac-128} (for the
coarse discretization) and in figure \ref{fig:vac-512} (for the fine
discretization). This problem seems to be especially challenging for
the cubic spline interpolation. For the coarse discretization we observe
an error of almost 50\% in the $L^{1}$ norm, which renders any interpretation
of the particle density function impossible. In addition, we observe
two unphysical peaks in the electric energy. 

On the other hand, none of these difficulties appear for the semi-Lagrangian
discontinuous Galerkin scheme. With respect to the $L^{1}$ norm the
second order dG method is most favorable although that method only
conserves energy up to an accuracy of $10^{-3}$ (compared to $10^{-5}$
for the higher order methods). This favorable behavior is despite
the fact that entropy conservation is somewhat better for the cubic
spline approximation and the $L^{2}$ norm conservation is significantly
better.

Let us continue to discuss the cubic spline interpolation. For this
method the entropy and the $L^{2}$ norm show an oscillating behavior
in time. Clearly this is a numerical artifact as has been discussed
in section \ref{sec:conservation}. None of these artifacts are present
for the discontinuous Galerkin approach (as is expected based on the
theoretical considerations in section \ref{sec:conservation}).

\subsection{Linear Landau damping\label{sec:ll}}

The linear Landau damping problem is just the strong Landau damping
introduced earlier but with $\alpha$ chosen sufficiently small such
that nonlinear effects are small. In this section we consider $\alpha=10^{-2}$.
We have postponed this classic problem until now because in some sense
the problem is very easy while in another sense it is very difficult.
The numerical results for the coarse grid are shown in figure \ref{fig:ll-128}.
All of the conserved quantities look very reasonable. In fact, $L^{1}$
is conserved up to machine precision and the $L^{2}$ norm and the
entropy are conserved up to approximately an error of $10^{-5}$ (by
far the lowest error among the problems considered so far). Nevertheless,
the result of the numerical simulation is completely wrong (even considering
a qualitative analysis). From analytical results it is well known
that the electric energy decays exponentially as a function in time.
What we observe for the numerical simulation, however, is that it
oscillates. This is the so-called recurrence effect; a numerical artifact
which implies that the number of grid points have to be chosen proportional
to the final time in order to obtain reasonable results. This was
already recognized in \citet{cheng1976} and is the reason why despite
the favorable conservation properties the linear Landau damping problem
is difficult for any numerical scheme.

Unfortunately, a number of interesting problems in plasma physics
(such as plasma echos; see \citet{hou2011} or \change{\cite{galeotti2005} for a good exposition})
suffer from the recurrence effect. The problem can be overcome to
some extend by filtering high frequencies in phase space (a technique
referred to as filamentation filtration). In fact, a significant body
of literature has been developed on this topic \citep{klimas1994,eliasson2002,einkemmer2014rec}. However, even though,
such methods can alleviate the problem to some extend, in many situations
a fine velocity discretization is still required in order to resolve
the interesting dynamics of the system. However, for such a fine discretization
the conservation behavior is even better than for the relative coarse
discretization used here.

At this point one might rightfully object that there is no guarantee
that filamentation filtration will respect the invariants (even if
they are conserved by the basic numerical scheme). Thus it seems prudent
to investigate the effect of such a modification to the algorithm
outlined in section \ref{sec:eq-numerical-methods}. The results for
the method introduced in \citet{einkemmer2014rec} are shown in figure
\ref{fig:ll-128}. There is some effect on the entropy and the $L^{1}$
norm which, however, are still conserved up to a accuracy of $10^{-2}$.
However, despite the fact that we essentially remove high frequencies
neither the mass nor any of the conserved quantities, except for the
$L^{1}$ norm and the entropy, are adversely affected. This behavior
is due to the fact that the magnitude of the high frequencies is very
small and thus negative values appear. Since we cut off a couple of
frequencies this averages out and leaves the mass intact (while affecting
the $L^{1}$ norm).

Before proceeding, let us return to the question of how coarse a space
discretization is sufficient in order to obtain reasonable results.
To do this we consider a model of the plasma echo phenomenon. That
is, we prescribe the initial value
\[
f(0,x,v)=\frac{1}{\sqrt{2\pi}}\mathrm{e}^{-v^{2}/2}(1+\alpha\cos k_{1}x)
\]
with $\alpha=10^{-3}$ and $k_{1}=12\pi/100$. This initial perturbation
is damped away by Landau damping. Then at time $t=200$ we excite
a second perturbation by adding 
\[
\frac{\alpha}{\sqrt{2\pi}}\mathrm{e}^{-v^{2}/2}\cos k_{2}x,
\]
where $k_{2}=25\pi/100$, to the particle density function. This perturbation
is also damped away by Landau damping. However, all the information
is still stored in the particle density function and due to constructive
interference an echo appears at $t=400$. In fact, multiple echoes
can be observed. We will investigate how well the first echo is resolved
by the semi-Lagrangian discontinuous Galerkin method. The numerical
results are shown in figure \ref{fig:echo}. We note that in this
situation increasing the order in the velocity direction, while keeping
the number of degrees of freedom constant, significantly degrades
the numerical solution. This is particular prominent for approximations
of order three and higher, where the plasma echo is barely recognizable.
In the spatial direction the results are comparable no matter what
order is chosen. Note, however, that decreasing the number of degrees
of freedom rapidly diminishes the numerical solution and the wave
echo disappears. Thus, there is no point in studying the conservation
properties at low resolution in the spatial direction as a certain
number of cells are required even if one is only interested in resolving
the first echo accurately.

\subsection{Long time behavior (Nonlinear Landau)}

We will now investigate the long time behavior for the semi-Lagrangian
discontinuous Galerkin approach. To that end we integrate the Vlasov--Poisson
equations until $T=2\cdot10^{4}$ for the classic nonlinear Landau
damping problem (as stated in section \ref{sub:nl}). The results
are shown in figure \ref{fig:nl-5k} and compare the second order
method with the fourth order method. Note that even though the conservation
properties are quite similar (especially with respect to the $L^{2}$
norm and the entropy) the second order method produces qualitatively
wrong results (the electric energy decays exponentially in time).
Note that this rather different behavior is not expected based on
the numerical results that were obtained up to final time $T=400$
(see section \ref{sub:nl}). On the other hand, the error in all the
conserved quantities is stable over this time interval.

\subsection{Long time behavior (bump-on-tail instability)}

In this section we investigate the long time behavior of the semi-Lagrangian
discontinuous Galerkin approach for the bump-on-tail instability by
integrating the Vlasov--Poisson equation until $T=2\cdot10^{4}$.
The numerical results are shown in figure \ref{fig:bot-5k}. Also
in this case the second order method produces qualitatively wrong
results. However, the conservation properties are also markedly worse
for the second order method as compared to the fourth and sixth order
methods. This is especially true for the entropy and $L^{2}$ norm
(which measures diffusion) and we are thus not surprised that the
second order method washes out the vortex in phase space. We also
observe that the error in the current is quite large (conserved only
up to approximately $10^{-2}$). This is true independent of the order
of the method. However, this behavior is not detrimental to the other
conserved quantities nor is it detrimental to the qualitative properties
of the numerical solution.

\subsection{Vortex merging in the bump-on-tail instability}

In the previous section only one vortex appears in the bump-on-tail
instability. However, for a certain regime of parameters multiple
vortices can be observed in phase space. The difficulty for a numerical
scheme is then to preserve these vortices as long as possible. However,
most numerical schemes eventually lead to a secondary instability
that results in the three vortices merging into a single one. In the
following we will consider the problem from \citet{crouseilles2011}
for which three vortices are observed in phase space. The numerical
results are shown in figure \ref{fig:three-vortices}. We note that
vortex merging starts at approximately $t=650$ for the fourth order
discontinuous Galerkin method and at approximately $t=1200$ for the
spline interpolation. The sixth order discontinuous Galerkin schemes
performs best in this test (vortex merging starts at approximately
$t=1250$).

\section{Conclusion\label{sec:Conclusion}}

We conclude that the semi-Lagrangian discontinuous Galerkin scheme
is certainly competitive compared to the cubic spline interpolation
with respect to the invariants considered in this work. The latter
seems to have a slight edge with respect to energy conservation (which,
however, is always preserved up to an accuracy of $10^{-4}$ or better).
On the other hand, the discontinuous Galerkin scheme outperforms the
cubic spline interpolation with respect to the conservation of the
$L^{1}$ norm (i.e.~ positivity preservation) in all tests. This
is particularly apparent in one test (the expansion into a uniform
ion background) where the cubic spline interpolation results in over
and undershoots that renders a physical interpretation of the particle
density function impossible. 

The second order variant of the discontinuous Galerkin scheme is usually
quite diffusive. However, using the fourth or a higher order variant
remedies this deficiency. Based on the conservation of the invariants
the fourth order method is sufficient for most problems. However,
for the bump-on-tail instability we observed a significant improvement,
both in conservation as well as in long time behavior (onset of vortex
merging), for the sixth order scheme. In general, the numerical results
obtained emphasize the efficiency of high order methods (even though
those methods reduce the number of cells present in the numerical
simulation and this behavior is not expected based on the regularity
of the solution). The exception to this rule is the plasma echo phenomenon
(where increasing the order actually degrades the numerical solution). 

In addition, we have seen that transferring the semi-Lagrangian scheme
to an equidistant grid (in order to solve Poisson's equation) does
not result in any issues with respect to conservation of the invariants
and long time behavior. The same is true for the filamentation filtration
approach considered in this paper.

\section*{Acknowledgements}
The computational results presented have been achieved using the Vienna Scientific Cluster (VSC).

\section*{}

\bibliographystyle{jpp}
\bibliography{vlasov-conservation}

\begin{figure}
\begin{centering}
\includegraphics[width=14cm]{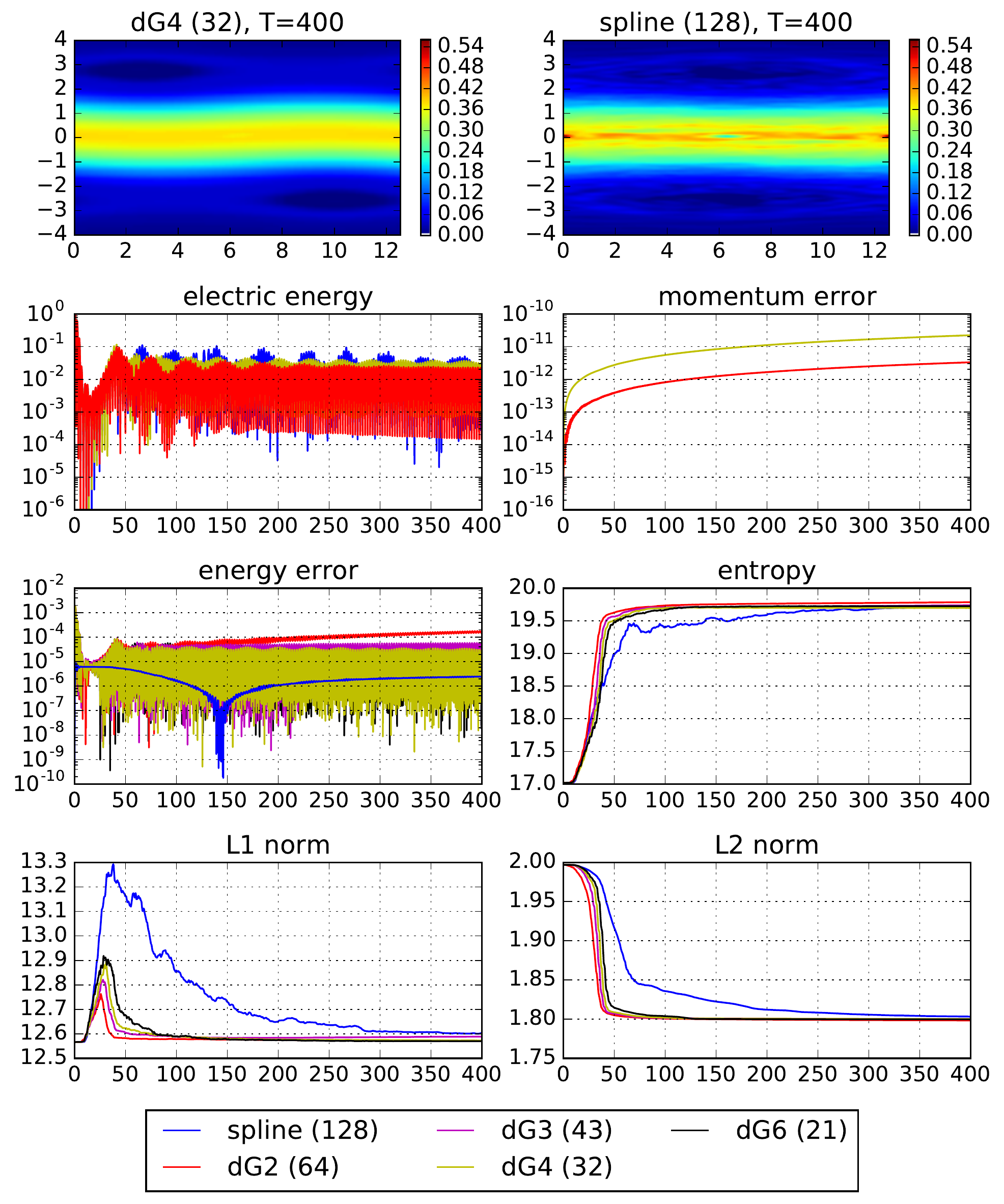}
\par\end{centering}

\caption{This figure shows the particle density $f$ at $T=400$ and the time
evolution of the electric energy for the nonlinear Landau damping
problem. In addition, the error in the current, energy, entropy, $L^{1}$
norm, and $L^{2}$ norm are shown. For all numerical schemes $128$
degrees of freedom are employed per space dimension. The order of
the discontinuous Galerkin (dG) method is indicated and the number
of cells are given in parenthesis. \label{fig:nl-128}}
\end{figure}

\begin{figure}
\begin{centering}
\includegraphics[width=14cm]{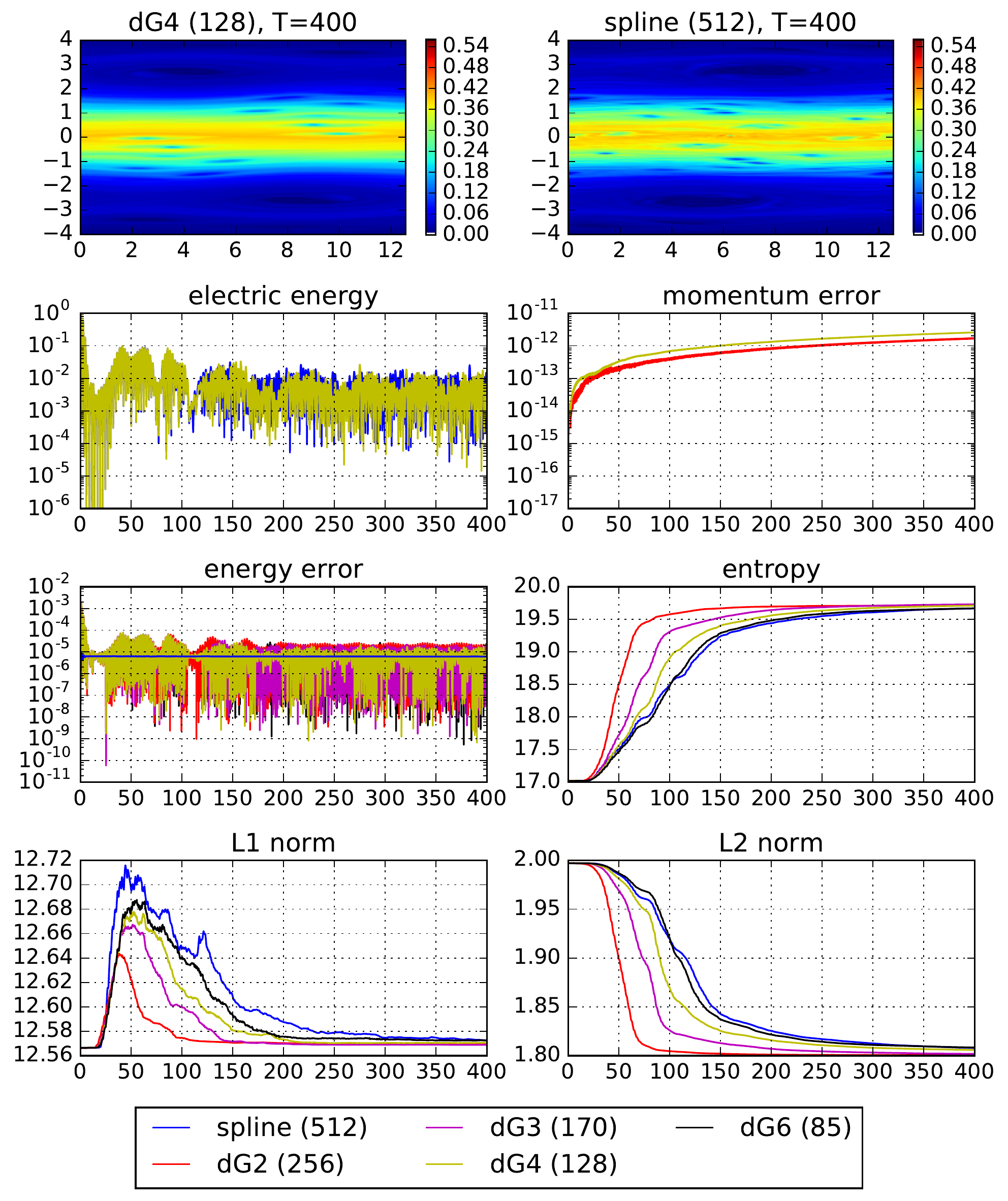}
\par\end{centering}

\caption{This figure shows the particle density $f$ at $T=400$ and the time
evolution of the electric energy for the nonlinear Landau damping
problem. In addition, the error in the current, energy, entropy $L^{1}$,
and $L^{2}$ norm are shown. For all numerical schemes $512$ degrees
of freedom are employed per space dimension. The order of the discontinuous
Galerkin (dG) method is indicated and the number of cells are given
in parenthesis. \label{fig:nl-512}}
\end{figure}

\begin{figure}
\begin{centering}
\includegraphics[width=14cm]{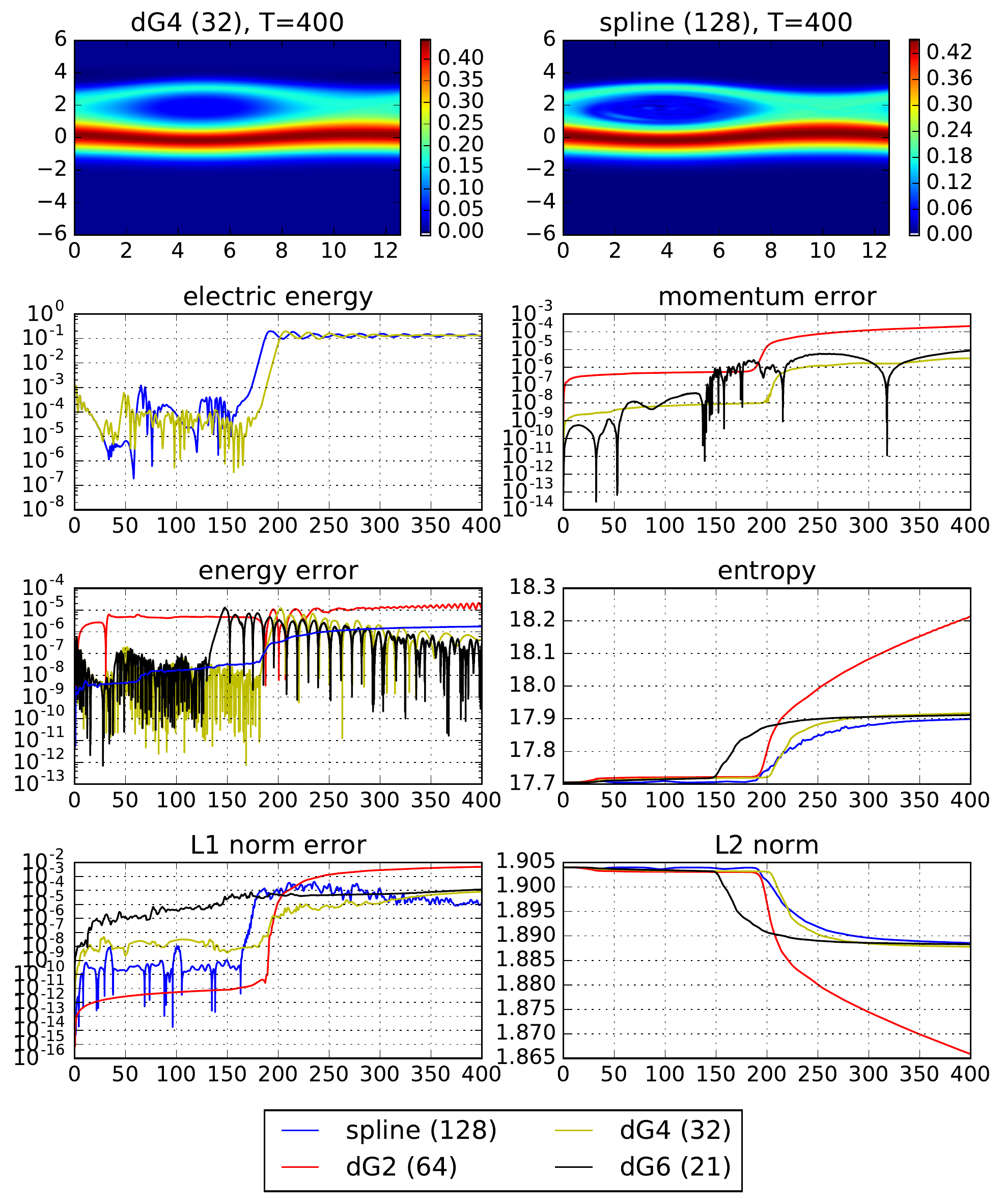}
\par\end{centering}

\caption{This figure shows the particle density $f$ at $T=400$ and the time
evolution of the electric energy for the bump-on-tail instability.
In addition, the error in the current, energy, entropy, $L^{1}$ norm,
and $L^{2}$ norm are shown. For all numerical schemes $128$ degrees
of freedom are employed per space dimension. The order of the discontinuous
Galerkin (dG) method is indicated and the number of cells are given
in parenthesis. \label{fig:bot-128}}
\end{figure}

\begin{figure}
\begin{centering}
\includegraphics[width=14cm]{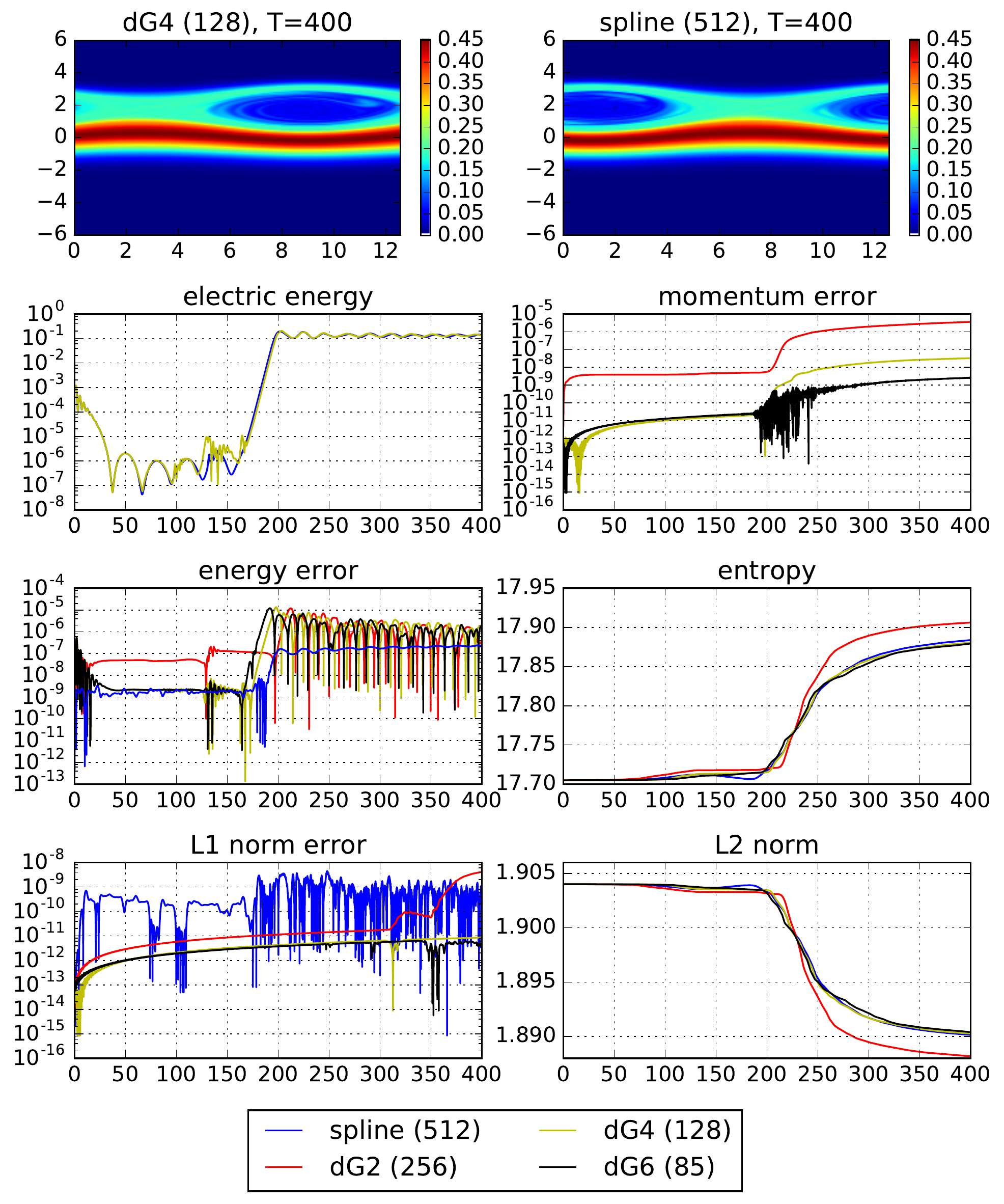}
\par\end{centering}

\caption{This figure shows the particle density $f$ at $T=400$ and the time
evolution of the electric energy for the bump-on-tail instability.
In addition, the error in the current, energy, entropy $L^{1}$, and
$L^{2}$ norm are shown. For all numerical schemes $512$ degrees
of freedom are employed per space dimension. The order of the discontinuous
Galerkin (dG) method is indicated and the number of cells are given
in parenthesis. \label{fig:bot-512}}
\end{figure}

\begin{figure}
\begin{centering}
\includegraphics[width=14cm]{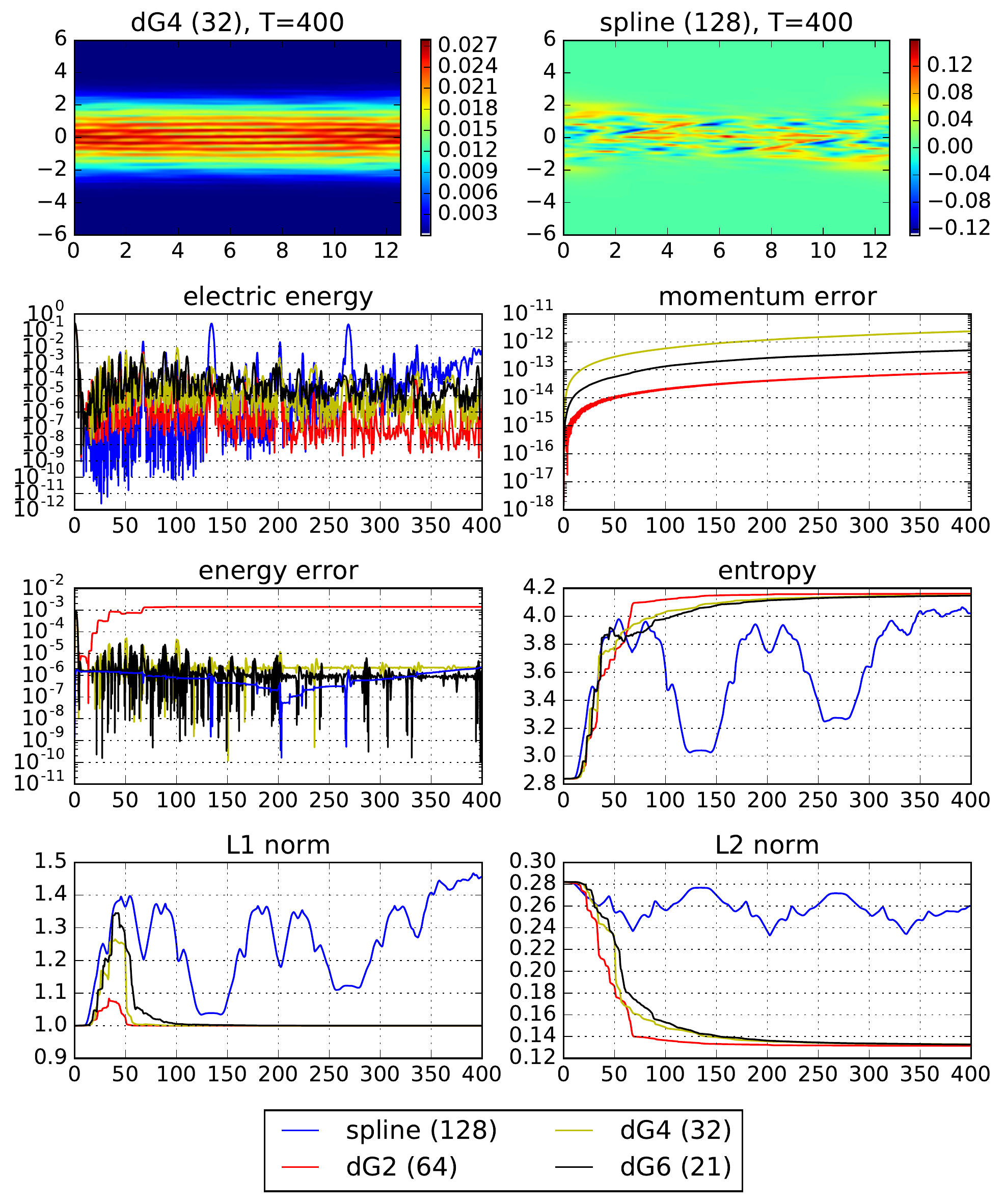}
\par\end{centering}

\caption{This figure shows the particle density $f$ at $T=400$ and the time
evolution of the electric energy for the expansion into a uniform
ion background. In addition, the error in the current, energy, entropy,
$L^{1}$ norm, and $L^{2}$ norm are shown. For all numerical schemes
$128$ degrees of freedom are employed per space dimension. The order
of the discontinuous Galerkin (dG) method is indicated and the number
of cells are given in parenthesis.\label{fig:vac-128}}
\end{figure}

\begin{figure}
\begin{centering}
\includegraphics[width=14cm]{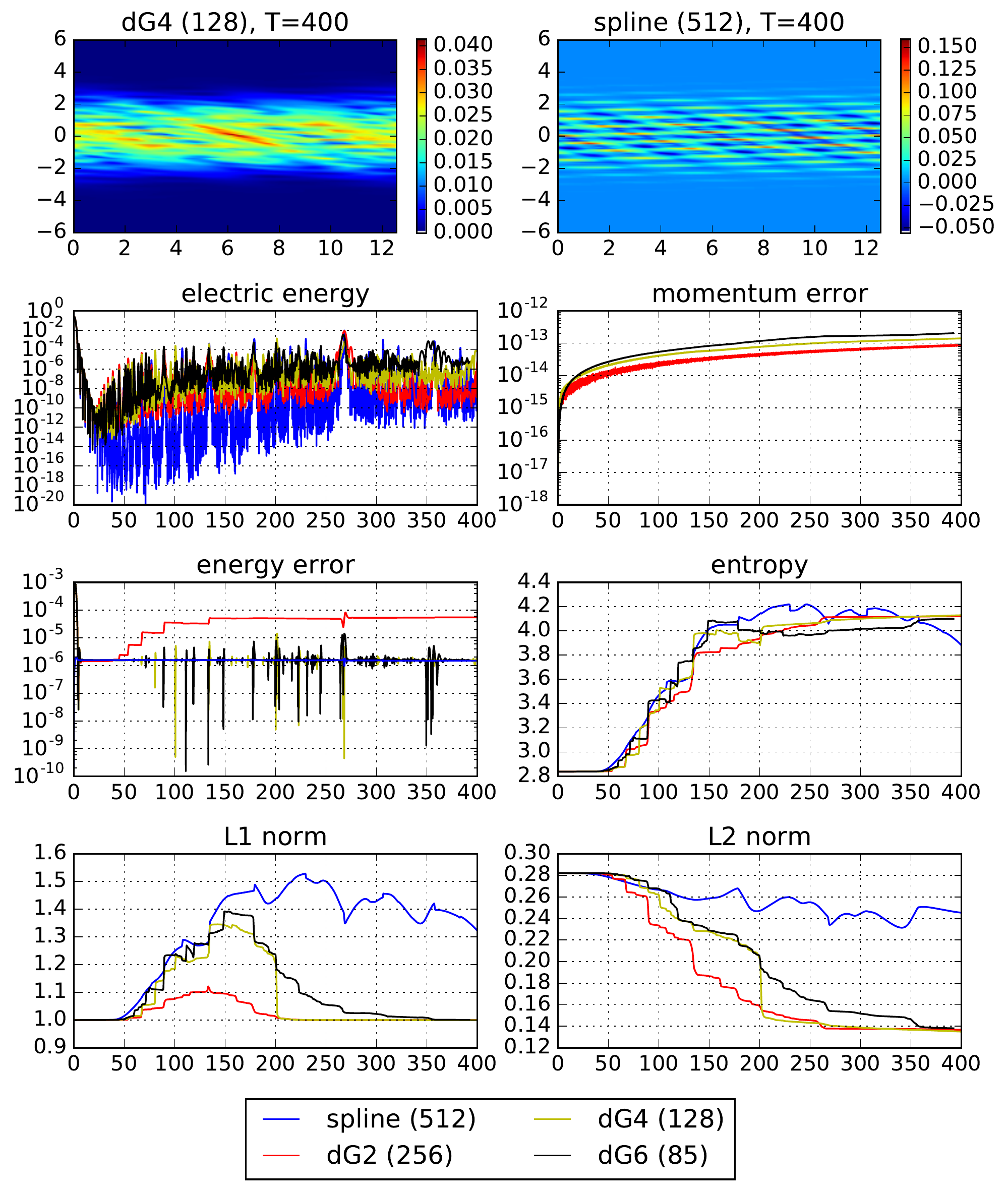}
\par\end{centering}

\caption{This figure shows the particle density $f$ at $T=400$ and the time
evolution of the electric energy for the expansion into a uniform
ion background. In addition, the error in the current, energy, entropy,
$L^{1}$ norm, and $L^{2}$ norm are shown. For all numerical schemes
$512$ degrees of freedom are employed per space dimension. The order
of the discontinuous Galerkin (dG) method is indicated and the number
of cells are given in parenthesis.\label{fig:vac-512}}
\end{figure}

\begin{figure}
\begin{centering}
\includegraphics[width=14cm]{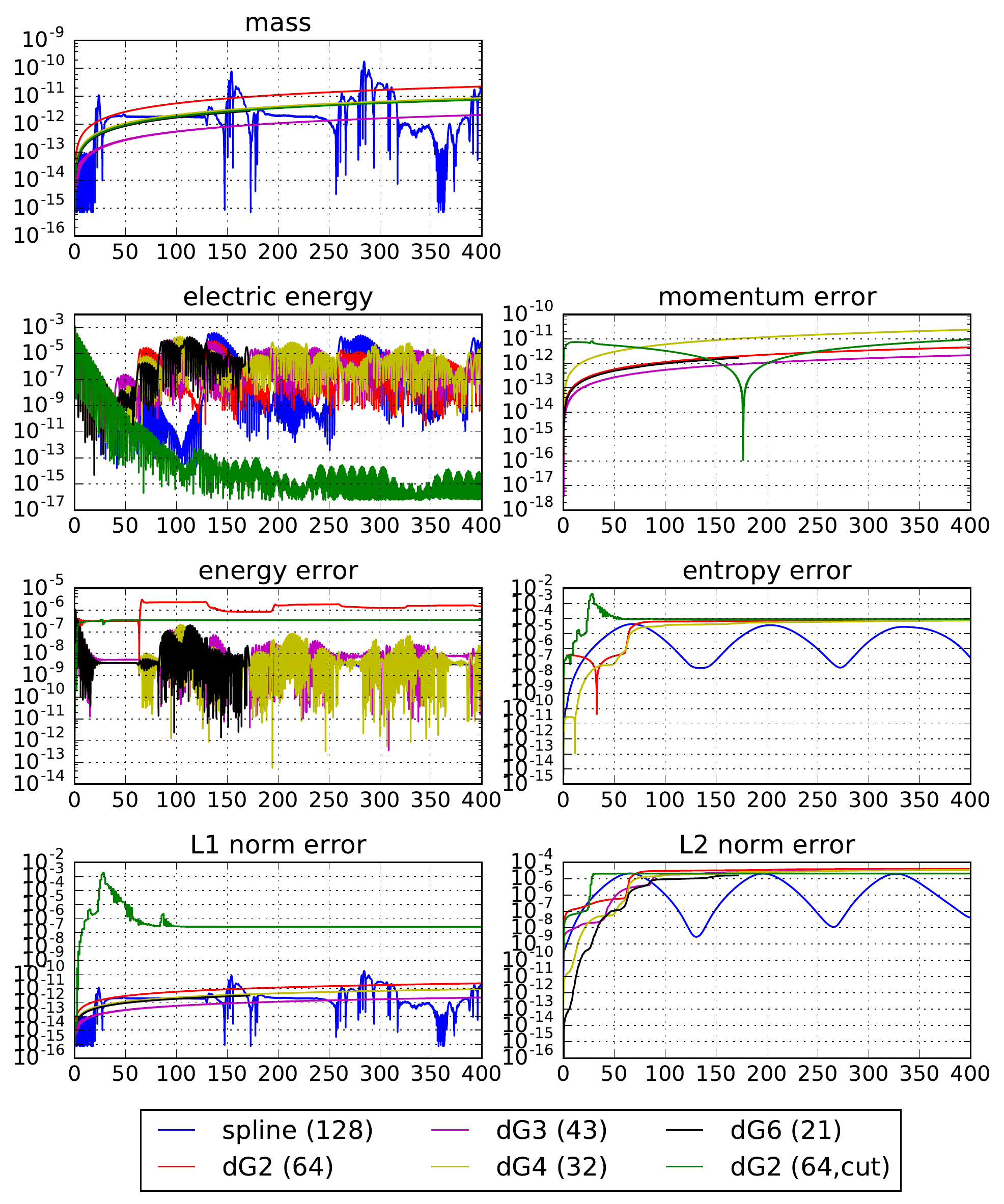}
\par\end{centering}

\caption{This figure shows the time evolution of the electric energy for the
linear Landau damping. In addition, the error in the current, energy,
entropy, $L^{1}$ norm, and $L^{2}$ norm are shown. For all numerical
schemes $128$ degrees of freedom are employed per space dimension.
The order of the discontinuous Galerkin (dG) method is indicated and
the number of cells are given in parenthesis.\label{fig:ll-128}}
\end{figure}

\begin{figure}
\begin{centering}
\includegraphics[width=14cm]{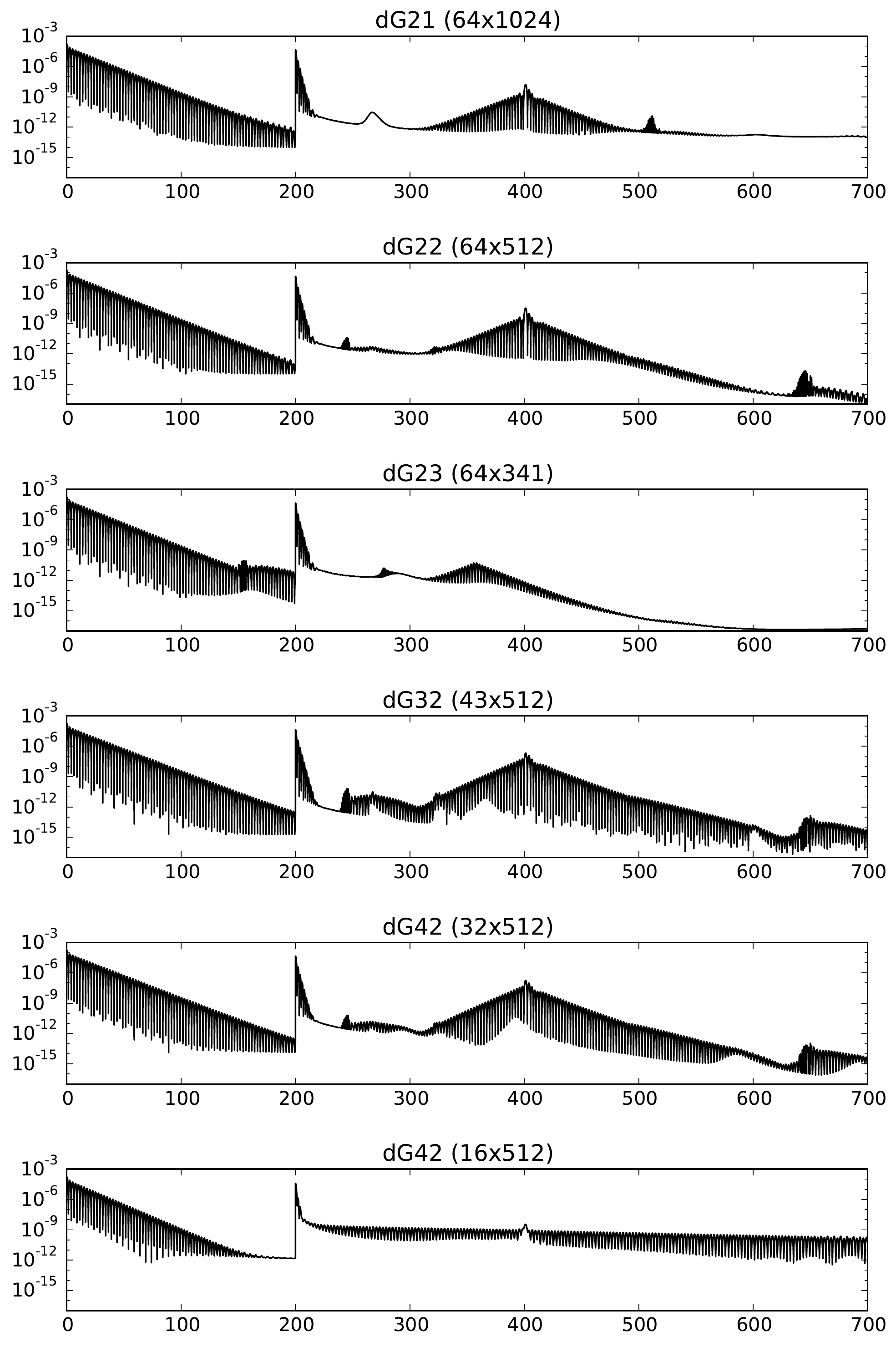}
\par\end{centering}

\caption{This figure shows the evolution of the electric energy for the plasma
echo problem. The order of the discontinuous Galerkin (dG) method
is indicated (dGab denotes a method with order a in the $x$-direction
and order b in the $v$-direction) and the number of cells are given
in parenthesis.\label{fig:echo}}
\end{figure}

\begin{figure}
\begin{centering}
\includegraphics[width=14cm]{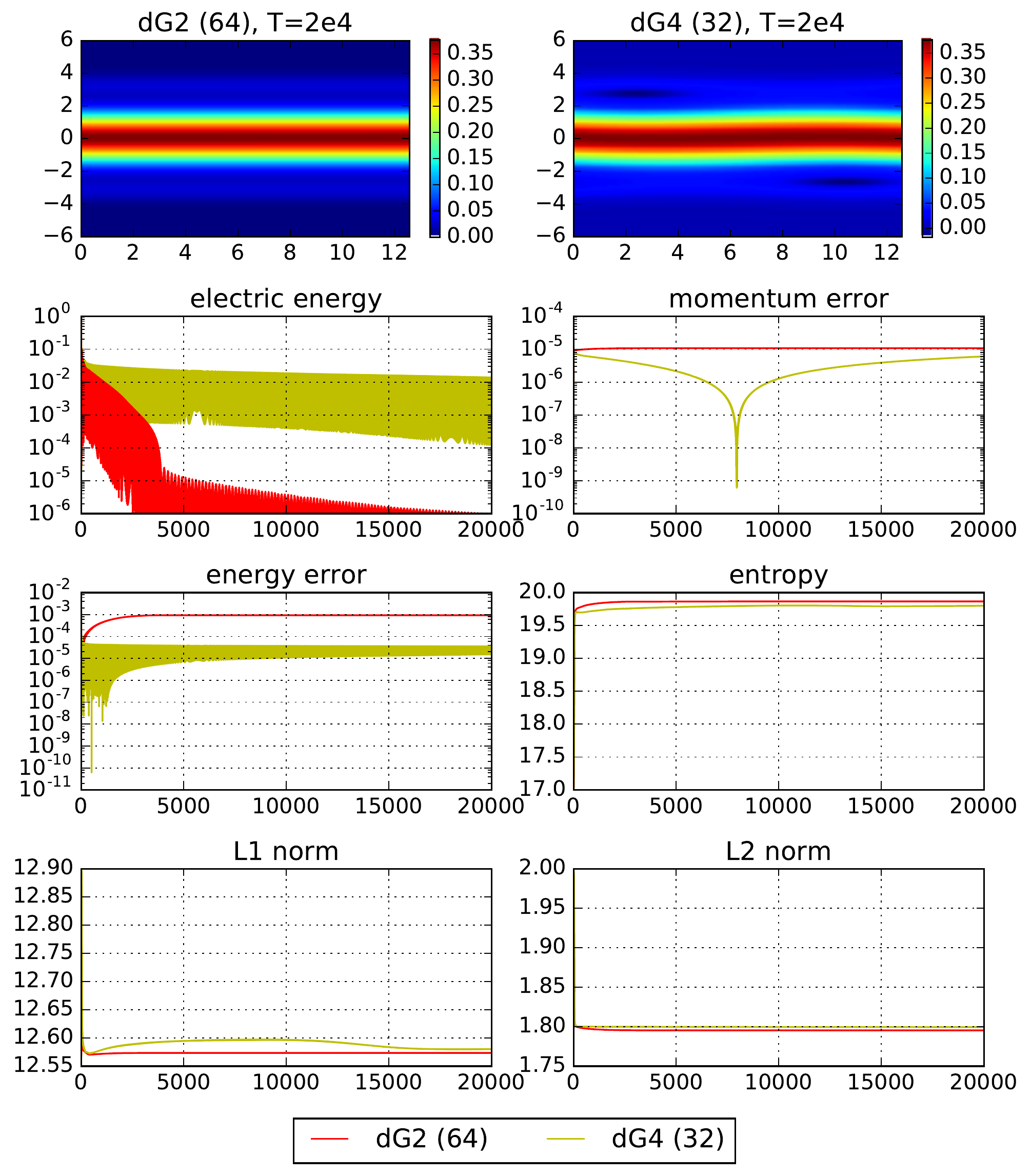}
\par\end{centering}

\caption{This figure shows the particle density $f$ at $T=2\cdot10^{4}$ and
the time evolution of the electric energy for the nonlinear Landau
damping. In addition, the error in the current, energy, entropy, $L^{1}$
norm, and $L^{2}$ norm are shown. For all numerical schemes $128$
degrees of freedom are employed per space dimension. The order of
the discontinuous Galerkin (dG) method is indicated and the number
of cells are given in parenthesis.\label{fig:nl-5k}}
\end{figure}

\begin{figure}
\begin{centering}
\includegraphics[width=14cm]{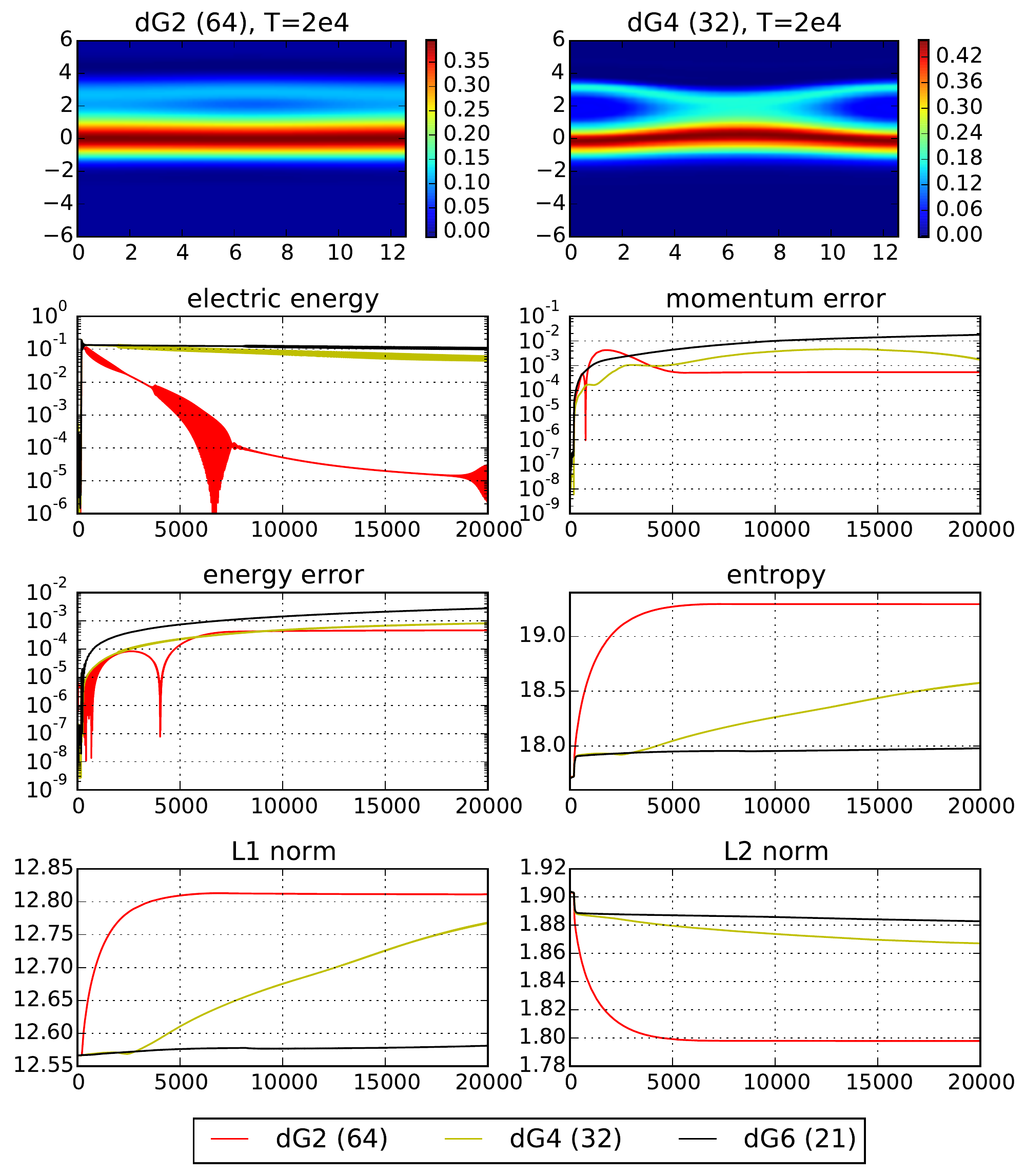}
\par\end{centering}

\caption{This figure shows the particle density $f$ at $T=2\cdot10^{4}$ and
the time evolution of the electric energy for the bump-on-tail instability.
In addition, the error in the current, energy, entropy, $L^{1}$ norm,
and $L^{2}$ norm are shown. For all numerical schemes $128$ degrees
of freedom are employed per space dimension. The order of the discontinuous
Galerkin (dG) method is indicated and the number of cells are given
in parenthesis.\label{fig:bot-5k}}
\end{figure}

\begin{figure}
\begin{centering}
\includegraphics[width=14cm]{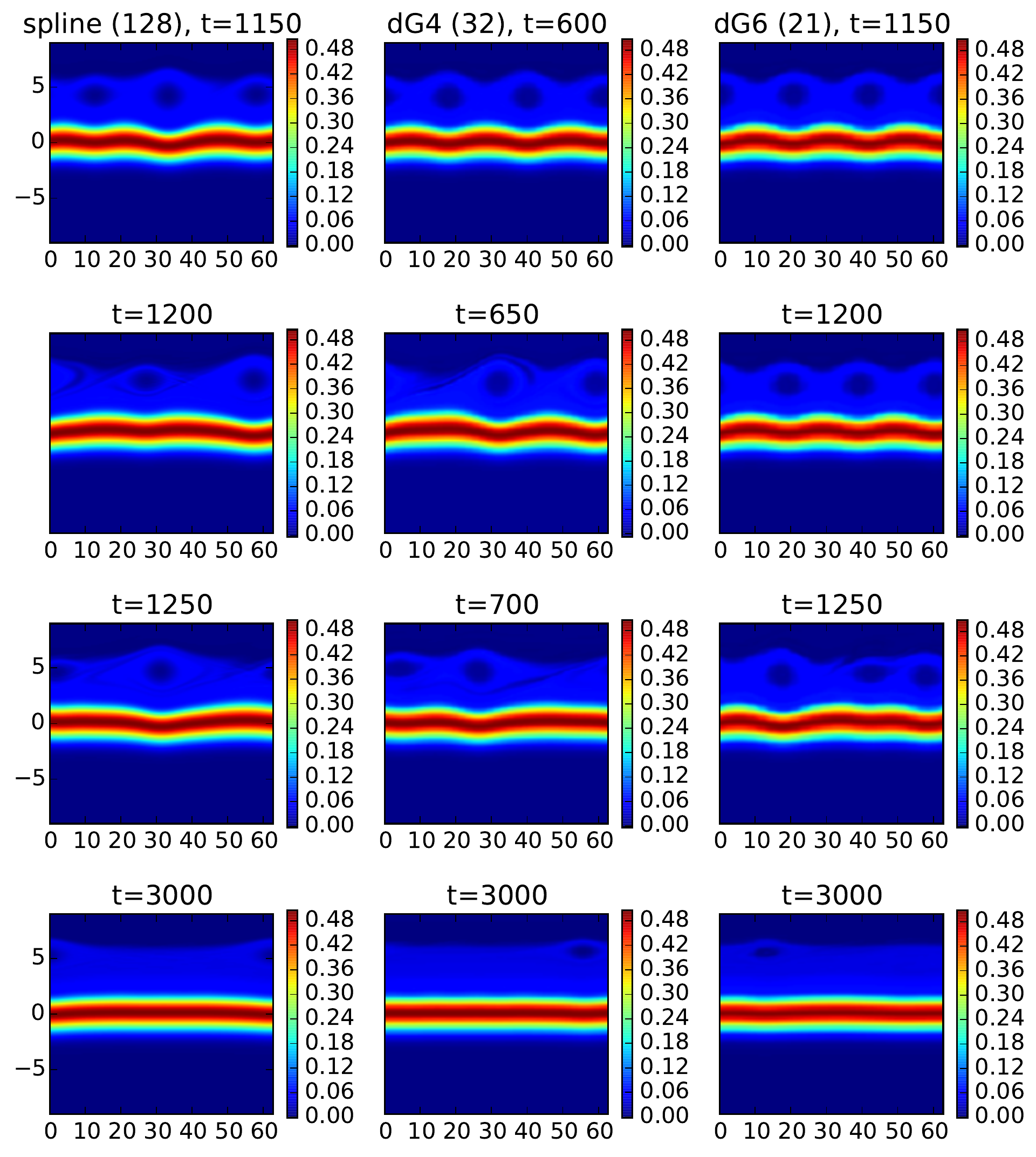}\caption{This figure shows the particle density $f$ at different times for
the three vortex bump-on-tail instability. For all numerical schemes
$128$ degrees of freedom are employed per space dimension. The order
of the discontinuous Galerkin (dG) method is indicated and the number
of cells are given in parenthesis.\label{fig:three-vortices}}

\par\end{centering}

\end{figure}

\end{document}